\pdfoutput=1
\documentclass[11pt,oneside,reqno]{amsart}

\usepackage{amsthm}
\usepackage{amssymb}
\usepackage{amsmath}
\usepackage{amscd}
\usepackage{esint}
\usepackage{enumitem}
\usepackage{bbm}
\usepackage{mathrsfs}
\usepackage[top=1.1in, bottom=1in, left=1in, right=1in]{geometry}

\usepackage{cite}

\usepackage[dvipsnames]{xcolor}
\usepackage[colorlinks=true, pdfstartview=FitV, linkcolor=blue, citecolor=blue, urlcolor=blue]{hyperref}
\hypersetup{
  pdftitle={Onsager-Type Energy Equality and Prodi--Serrin Uniqueness for Nernst--Planck Fluid Systems},
  pdfauthor={Ruimeng Hu, Quyuan Lin, and Qirui Peng},
  pdfkeywords={Nernst--Planck equations; Euler equations; Navier--Stokes equations; Onsager conjecture; energy equality; Prodi--Serrin condition}
}

\newcommand{\lb}{\varLambda}

\def \<{\langle}
\def \>{\rangle}

\newcommand{\bg}{\begin{equation}}
\newcommand{\ed}{\end{equation}}
\newcommand{\bga}{\begin{eqnarray}}
\newcommand{\eda}{\end{eqnarray}}

\def\cbdu{\par{\raggedleft$\Box$\par}}

\newtheorem {Theorem}  {Theorem}
\newtheorem {Corollary}[Theorem]{\bf Corollary}

\numberwithin{Theorem}{section}

\newtheorem {Lemma}[Theorem]  {Lemma}

\theoremstyle{definition}
\newtheorem{Definition}[Theorem]{Definition}
\theoremstyle{remark}
\newtheorem{Remark}[Theorem]{\bf Remark}

\def \l{\lambda}
\expandafter\chardef\csname pre amssym.def
at\endcsname=\the\catcode`\@ \catcode`\@=11
\def\undefine#1{\let#1\undefined}
\def\newsymbol#1#2#3#4#5{\let\next@\relax
 \ifnum#2=\@ne\let\next@\msafam@\else
 \ifnum#2=\tw@\let\next@\msbfam@\fi\fi
 \mathchardef#1="#3\next@#4#5}
\def\mathhexbox@#1#2#3{\relax
 \ifmmode\mathpalette{}{\m@th\mathchar"#1#2#3}%
 \else\leavevmode\hbox{$\m@th\mathchar"#1#2#3$}\fi}
\def\hexnumber@#1{\ifcase#1 0\or 1\or 2\or 3\or 4\or 5\or 6\or 7\or 8\or
 9\or A\or B\or C\or D\or E\or F\fi}

\font\teneufm=eufm10 \font\seveneufm=eufm7 \font\fiveeufm=eufm5
\newfam\eufmfam
\textfont\eufmfam=\teneufm \scriptfont\eufmfam=\seveneufm
\scriptscriptfont\eufmfam=\fiveeufm

\catcode`\@=\csname pre amssym.def at\endcsname

\newcounter{remark}
\numberwithin{equation}{section}
\numberwithin{figure}{section}

\def \grad {\nabla}

\def \onethird {\frac{1}{3}}
\def \mone  {{-1}}

\newcommand{\divv}{{\text {div}}\,}

\newcommand{\e}{\epsilon}

\newcommand{\om}{\omega}

\renewcommand{\l}{\lambda}

\renewcommand{\th}{\theta}

\newcommand{\s}{\sigma}

\newcommand{\R}{\mathbf{R}}

\newcommand{\les}{\lesssim}
\renewcommand{\div}{\mbox{div}}

\newcommand{\Bb}{{\mathcal B}}

\newcommand{\Ff}{{\mathcal F}}

\newcommand{\Ss}{{\mathcal S}}

\newcommand{\barphi}{\overline{\Phi}}
\newcommand{\barrho}{\overline{\rho}}
\newcommand{\barsgm}{\overline{\sigma}}

\newcommand{\IdM}{\textit{Id}}

\newcommand{\wth}{\widetilde{h}}

\def  \R   {{\mathbb R}}
\def  \Z   {{\mathbb Z}}

\def  \N   {{\mathbb N}}

\def  \T   {{\mathbb T}}

\def  \haf  {{\frac{1}{2}}}
\def  \p   {\partial}

\def  \tr   {\operatorname{Tr}}

\def  \Dq    {\Delta_q}
\def  \Dj    {\Delta_j}

\def  \sumjtq  {\sum_{j=-1}^q}
\def  \sumjgtrq {\sum_{j > q}}

\def  \intthree   {\int_{\T^3}}

\def  \Ssq    {\Ss_q}
\def  \Ssj    {\Ss_j}

\newcommand\onenorm[1]{\lVert#1\rVert_{L^1(\T^3)}}
\newcommand\twonorm[1]{\lVert#1\rVert_{L^2(\T^3)}}
\newcommand\threenorm[1]{\lVert#1\rVert_{L^3(\T^3)}}
\newcommand\fournorm[1]{\lVert#1\rVert_{L^4(\T^3)}}

\newcommand\qnorm[1]{\Vert#1\Vert_{L^q(\T^3)}}

\newcommand\Linfnorm[1]{\Vert#1\Vert_{L^\infty(\T^3)}}

\newcommand\Honenorm[1]{\Vert#1\Vert_{H^1(\T^3)}}
\newcommand\Hsnorm[1]{\Vert#1\Vert_{H^s(\T^3)}}

\newcommand\threeNorm[1]{\left \lVert #1 \right \rVert_{L^3(\T^3)}}
\newcommand\sixNorm[1]{\left \lVert #1 \right \rVert_{L^6(\T^3)}}

\newcommand\LinfNorm[1]{\left \lVert#1\right \rVert_{L^\infty(\T^3)}}

\newcommand\threetwoNorm[1]{\left \lVert #1 \right \rVert_{L^{\frac{3}{2}}(\T^3)}}
\newcommand\sixfiveNorm[1]{\left \lVert #1 \right \rVert_{L^{\frac{6}{5}}(\T^3)}}

\newcommand{\Ltwoinner}[1]{\left < #1 \right >}

\def\build#1_#2^#3{\mathrel{\mathop{\kern 0pt#1}\limits_{#2}^{#3}}}

\begin{document}

\title[Onsager equality and Prodi--Serrin uniqueness for NP fluids]{Onsager-Type Energy Equality and Prodi--Serrin Uniqueness for Electrodiffusion Systems}

\author [Ruimeng Hu]{Ruimeng Hu}
\address{Department of Mathematics,  Department of Statistics and Applied Probability, University of California, Santa Barbara, CA 93106, USA}
\email{rhu@ucsb.edu}

\author [Quyuan Lin]{Quyuan Lin}
\address{School of Mathematical and Statistical Sciences, Clemson University, Clemson, SC 29634, USA}
\email{quyuanl@clemson.edu}

\author [Qirui Peng]{Qirui Peng}
\address{Department of Mathematics,  University of California, Santa Barbara, CA 93106, USA}
\email{qpeng9@ucsb.edu} 





\begin{abstract}
We study weak solutions of electrodiffusion systems coupling the Nernst--Planck equations with fluid models. First, for the three-dimensional Nernst--Planck--Euler system, we establish an Onsager-type criterion for the validity of the coupled kinetic-electrostatic energy balance. The energy equality is shown to hold for weak solutions whose velocity satisfies critical Besov regularity and a vanishing dyadic flux condition. Furthermore, assuming the corresponding Onsager-type regularity for the ionic concentrations, we also prove parabolic regularity, preservation of non-negativity of the concentrations, and the associated charge-density energy identity. Second, for the three-dimensional Nernst--Planck--Navier--Stokes system, we prove a Prodi--Serrin-type uniqueness criterion for Leray--Hopf solutions: uniqueness in the Leray--Hopf class holds whenever the velocity field lies in the Ladyzhenskaya--Prodi--Serrin class $L^p_tL^q_x$ with $2/p+3/q=1$ and $q>3$. These results extend energy-equality and weak--strong uniqueness principles from incompressible fluid dynamics to electrodiffusion models involving convection, diffusion, and self-consistent electrostatic forcing.

\bigskip


\end{abstract}

\maketitle

\section{Introduction}\label{Sec:Intro}

The Nernst--Planck equations are a continuum description of ionic electrodiffusion, going back to the foundational works of Nernst and Planck on diffusion and electromigration in electrolytes \cite{Nernst1888,Planck1890}.  For each charged species, the flux is driven by ordinary diffusion and by drift in the electrochemical potential.  Coupling this transport law to Poisson's equation for the self-consistent electric field gives the Poisson--Nernst--Planck system, a basic model in electrochemistry, semiconductor theory, membrane biophysics, ion channels, and micro- and nano-fluidics \cite{Rubinstein1990,NewmanAlyea2004,ChoiLui1995,Jasielec2021}.  When the electrolyte is immersed in a moving incompressible fluid, the fluid advects the ions and the charge density in turn exerts an electric body force on the fluid.  This feedback produces Nernst--Planck fluid models, in which transport, diffusion, electrostatic screening, and hydrodynamic motion are coupled at the PDE level.  The electrodiffusion model in an incompressible fluid for $N$ ionic species may be written as
\begin{equation*}
\begin{aligned}
    \partial_t c_i + u \cdot \nabla c_i
    &= \nabla \cdot \left(D_i(\nabla c_i+z_i c_i\nabla\Phi)\right),
    && i=1,\ldots,N, \\
    -\epsilon\Delta\Phi &= \rho=\sum_{i=1}^N z_i c_i, \\
    \partial_t u + u\cdot\nabla u - \nu\Delta u + \nabla p
    &= -(k_B T_k)\rho\nabla\Phi, \\
    \nabla\cdot u &=0,
\end{aligned}
\end{equation*}
where $c_i$ is the concentration of the $i$th ionic species, $z_i$ is its valence, $D_i$ is its diffusivity, $k_B$ is the Boltzmann constant, $T_k$ is the constant temperature, $\Phi$ is the electric potential, $\rho$ is the charge density, $u$ is the velocity field, $p$ is the pressure, $\epsilon$ is the dielectric constant, and $\nu\ge 0$ is the fluid viscosity.  The inviscid case $\nu=0$ corresponds to the Nernst--Planck--Euler (NPE) system, while $\nu>0$ gives the viscous Nernst--Planck--Navier--Stokes (NPNS) system.

The electrodiffusion systems have been studied intensively in the mathematical literature.
For the NPE system, the global existence of weak solutions was proved in \cite{IgnatovaShu2021}, and the global well-posedness and Gevrey regularity have been established in \cite{AbdoLeeWang2023}. For the viscous case, the existence of global weak solutions to the NPNS system has been established in \cite{FischerSaal2017,Schmuck2009}.
  In \cite{ConstantinIgnatova2019}, the authors considered the 2D bounded domain case with various physical boundary conditions for $c_i$ and demonstrated the existence and uniqueness of global strong solutions and their convergence to Boltzmann steady states. For strong solutions in 3D, \cite{ConstantinIgnatovaLee2021Far} studied the case of two ionic species and many ionic species having equal diffusivities. In the context of Boltzmann states' stability, \cite{ConstantinIgnatovaLee2022Near} demonstrated nonlinear stability in both 2D and 3D bounded domains with certain boundary conditions. Some other related electrodiffusion models include Nernst--Planck--Boussinesq system \cite{AbdoHuLin2024NPB}, in which the non-isothermal setting was considered, and Nernst--Planck--Darcy system \cite{ignatova2022global,AbdoLeeWang2023}, where the fluid obeys Darcy's law. 

For the Onsager part of this paper, we focus on the NPE system with two ionic species, opposite valences, and equal ionic diffusivity.  Taking $N=2$, $z_1=1$, $z_2=-1$, $D_1=D_2=D$, and introducing the charge and total concentration variables $\rho=c_1-c_2$ and $\s=c_1+c_2$, the system becomes
\begin{subequations} \label{eq:NP_Euler}
\begin{align}
    \p_t \rho + u \cdot \nabla \rho &= D \left( \Delta \rho + \nabla \s \cdot \nabla \Phi +\s \Delta \Phi \right), \label{eq:NP_Euler_1}  \\  
    \p_t \s + u \cdot \nabla \s &= D \left( \Delta \s + \nabla \rho \cdot \nabla \Phi +\rho \Delta \Phi \right), \label{eq:NP_Euler_2} \\
    -\e \Delta \Phi &= \rho, \label{eq:NP_Euler_3}\\
    \p_t u + (u \cdot \nabla )u + \nabla p &= -(k_B T_k) \rho \nabla \Phi, \label{eq:NP_Euler_4} \\
    \nabla \cdot u &= 0.  \label{eq:NP_Euler_5}
\end{align}
\end{subequations}
System \eqref{eq:NP_Euler} is equipped with initial data $(\rho,\s, u)(\cdot,0) = (\rho_0,\s_0,u_0)$ and periodic boundary conditions.
For simplicity of notation, throughout the remainder of the paper we set
\[
k_B T_k = D = \epsilon = 1.
\]

The central quantity in the NPE system is the sum of the kinetic energy of the fluid and the electrostatic energy of the ionic field.  For any smooth solution $(u,\rho,\s)$ of \eqref{eq:NP_Euler}, it holds that
\begin{align}
& \twonorm{u(t)}^2 - \twonorm{u_0}^2  +   \twonorm{\nabla \Phi(t)}^2 - \twonorm{\nabla \Phi_0}^2  \notag \\
= \, & -2 \int_0^t \int_{\T^3} |\Delta \Phi|^2 dx d\tau - 2 \int_0^t \int_{\T^3} \s |\nabla \Phi|^2 dx d\tau, \label{eq:Energy_equality}
\end{align}
for $t\in[0,T]$. The motivation for the present work comes from Onsager's theory of ideal turbulence.  Onsager predicted that weak solutions of the incompressible Euler equations conserve kinetic energy if their spatial H\"older regularity is above $1/3$, while below this threshold they may dissipate energy anomalously even in the absence of viscosity \cite{Onsager1949}. The positive direction was proved in \cite{Eyink1994,CET94}, and later sharpened in Besov spaces by \cite{CCFS08}; the dissipative direction was completed through convex integration by \cite{Isett2018} and further refined by \cite{BDLSV2019}. This viewpoint is closely tied to anomalous dissipation in vanishing-viscosity limits and has analogues for advective-diffusive scalar equations, where one studies possible anomalous diffusion or anomalous scalar dissipation in the limit of vanishing diffusivity \cite{ColomboCrippaSorella2023}.  The NPE system is a natural electro-diffusive analogue of these questions: the Euler part carries the Onsager energy flux, while the Nernst--Planck part contains true diffusion and a nonlinear electrostatic drift.  Establishing energy equality for weak solutions to the NPE system therefore requires combining the commutator mechanism of the Onsager theory with the parabolic and elliptic cancellations specific to the Nernst--Planck coupling. 

The identity \eqref{eq:Energy_equality} records the exact exchange between kinetic energy and electrostatic energy, while the Nernst--Planck diffusion produces the two dissipative terms on the right-hand side.  For weak solutions, however, the above calculation is not automatic: the convective flux, the electrostatic forcing, and the self-consistent drift may all create defect measures unless additional regularity or cancellation is available.  One of the main goals of this paper is to identify natural conditions under which no such hidden energy defect occurs.  

The discussion of the viscous part in this paper is motivated by the classical Ladyzhenskaya--Prodi--Serrin theory for the three-dimensional Navier--Stokes equations.  Leray--Hopf weak solutions exist globally in time but uniqueness and regularity remain open in general.  The Serrin criterion asserts, roughly speaking, that if a weak solution belongs to a scale-invariant class $L^p(0,T;L^q)$ with $2/p+3/q=1$, then the solution is regular and unique in the appropriate weak class \cite{Prodi1959,Serrin1962,Ladyzhenskaya1967}.  This condition is one of the central weak--strong uniqueness mechanisms for Navier--Stokes, and it is naturally formulated at the scaling level of the nonlinear transport term; see, for example, the modern discussions in \cite{BV22,CL22}.  For the NPNS system, the momentum equation is driven by the electrostatic force $\rho\nabla\Phi$, and the ionic concentrations solve drift--diffusion equations driven by both $u$ and $\nabla\Phi$.  A Serrin-type theory in this setting must therefore control not only the usual Navier--Stokes relative energy terms, but also the cross kinetic--electrostatic terms and the Nernst--Planck dissipation.  This motivates our development of a Serrin criterion for NPNS Leray--Hopf solutions. 

We now summarize the main results of the paper.  First, for the NPE system, we prove an Onsager-type energy equality for weak solutions under a critical Besov regularity and vanishing dyadic flux condition on the velocity; see Theorem \ref{Thm:Energy_equality_NPEuler}.  The proof is based on a Littlewood--Paley commutator analysis in the spirit of \cite{CET94,CCFS08}, but it also uses the elliptic relation $-\epsilon\Delta\Phi=\rho$ and the precise cancellations in the Nernst--Planck drift.  Second, under corresponding Besov assumptions on the ionic concentrations, we prove the energy identity for the charge density and total ionic concentrations, and obtain non-negativity of the original ionic concentrations; see Theorem \ref{Thm:non-negativity_NPEuler}.  Finally, for the NPNS system, we establish a Serrin-type uniqueness theorem for Leray--Hopf type solutions; see Theorem \ref{Thm:Serrin_NPNL}.  The argument combines a relative energy identity with a justification of the fundamental theorem of calculus for the cross kinetic and electrostatic energies under the Serrin condition.

The rest of the paper is organized as follows.  Section \ref{sec:Preliminaries} recalls the Littlewood--Paley theory and Besov spaces used throughout the paper.  Section \ref{sec:NP_Euler_Energy_equality_condition} studies weak solutions of the NPE system and proves the Onsager-type energy equalities and non-negativity results.  Section \ref{sec:NPNL} treats the NPNS system and proves the Serrin-type uniqueness criterion.

\section{Preliminaries}\label{sec:Preliminaries}
We collect some technical tools that are used throughout the paper in the study of energy equalities for fluid equations.
\subsection{Littlewood--Paley decomposition}\label{LP_Intro}
We recall the Littlewood--Paley theory on $\T^d$. The reader is referred to the standard textbooks by Bahouri, Chemin, and Danchin~\cite{bahouri2011fourier} and by Grafakos~\cite{grafakos2008classical} for more detailed discussions of this topic. We start by fixing a non-negative radial function $\chi \in C_0^\infty(\R^d)$ such that
\begin{equation} \label{def:chi}
\chi(\xi) :=
\begin{cases}
1, & \text{for } |\xi| \le \tfrac{3}{4},\\
0, & \text{for } |\xi| \ge 1.
\end{cases}
\end{equation}
Next we define
\begin{equation}\label{def:varphi}
\varphi(\xi) := \chi(\xi/2) - \chi(\xi),
\end{equation}
and let
\begin{equation*}
\varphi_q(\xi) :=
\begin{cases}
\varphi(\l_q^{-1}\xi), & q \ge 0,\\
\chi(\xi), & q = -1,
\end{cases}
\end{equation*}
where $\l_q = 2^q$. In Fourier space, the family $\{\varphi_q\}_{q \ge -1}$ gives a dyadic partition of unity. For a tempered distribution $u$ on $\T^3$, its $q^{\mathrm{th}}$ Littlewood--Paley projection is defined by
\[
  \Delta_q u(x) := \sum_{k\in\Z^3} \hat{u}(k)\,\varphi_q(k)\,e^{ 2\pi ik \cdot x},
\]
where $\hat{u}(k)$ denotes the $k^{\mathrm{th}}$ Fourier coefficient of $u$. Then
\[
u = \sum_{q=-1}^\infty \Delta_q u,
\]
in the sense of distributions. For any $q \in \N$, we define the low-frequency cut-off operator
\[
\Ss_q u := \sum_{j=-1}^q \Delta_j u.
\]
For later use, set
\begin{equation}
h = \Ff^{-1} \varphi \ \ \text{and} \ \ \widetilde{h} = \Ff^{-1} \chi, \label{def:h_and_htilde}
\end{equation}
where $\Ff u(k):=\hat{u}(k)$. Hence
\begin{align}
    \Dq u &= \Ff^{-1} \left( \varphi_q \Ff u \right) = \l^d_q \int_{\T^d} h(\l_q y) u(x-y) dy, \ \ q \geq 0, \label{eq:Dq_u} \\
    \Delta_{-1} u &= \Ff^{-1} \left ( \chi \Ff u \right) = \int_{\T^d} \widetilde{h}(y) u (x-y) dy.  \label{eq:Dmone_u}
\end{align}

\noindent By \eqref{def:varphi} and \eqref{def:h_and_htilde},
\begin{equation}\label{eq:q_mode_cutoff}
\Ssq u = \Ff^{-1} \left( \chi(\l^{-1}_{q+1} \cdot ) (\Ff u) (\cdot) \right) = \l^d_{q+1}\int_{\T^d} \widetilde{h}(\l_{q+1} y) u(x-y) dy .
\end{equation}

We also recall Bernstein's inequality for each block of the Littlewood--Paley decomposition.
\begin{Lemma}[Bernstein's inequality]\label{lemma:Bernstein}
Let $d$ be the spatial dimension, $r\geq s\geq 1$, and $k\geq 0$. Then, for every zero-mean tempered distribution $u$,
\bg\label{Bern1}
\lambda_q^{k}\|\Dq u\|_{L^r (\T^d)} \lesssim \|\nabla^k \Dq u\|_{L^r (\T^d)}\lesssim \lambda_q^{k+d(\frac{1}{s}-\frac{1}{r})}\|\Dq u\|_{L^s(\T^d)}.
\ed
For the low-frequency block $q=-1$, the left inequality is understood on the zero-mean subspace; the Poincar\'e inequality on this finite block gives the same bound with a constant independent of $q$.
\end{Lemma}

\subsection{Sobolev and Besov spaces}
For functions $f,g \in L^2(\T^3)$, the $L^2$-inner product is defined by
\[
\left<f,g \right> = \int_{\T^3} f(x) g(x) dx.
\]
For $f \in L^2(\T^3)$, we write its Fourier series as
\[
f(x) = \sum_{k \in \Z^3} \widehat{f}_k e^{2\pi i k \cdot x}, \ \ \ \widehat{f}_k := \int_{\T^3} f(x) e^{-2\pi i k \cdot x} dx,
\]
and call $\widehat{f}_k$ the $k^{\mathrm{th}}$ Fourier coefficient.
For $s \in \R$, we write $H^s(\T^3)$ for the Sobolev space on $\T^3$ endowed with the norm
\[
\Hsnorm{f} := \biggl( \sum_{k \in \Z^3} \left(1 + |k|^2 \right)^s |\widehat{f}_k|^2 \biggr)^\haf.
\]
If $f$ has zero mean, then its $H^s$-norm is equivalent to the homogeneous seminorm $\dot{H}^s$ defined by
\[
\|f \|_{\dot{H}^s(\T^3)} := \biggl( \sum_{k \in \Z^3\setminus\{0\}} |k|^{2s} |\widehat{f}_k|^2 \biggr)^\haf.
\]
We also recall the dyadic characterization of the $H^s$-norm:
\[
\|f \|_{\dot{H}^s(\T^3)} \sim \left( \sum_{q=-1}^\infty \l^{2s}_q \twonorm{\Dq f}^2 \right)^\haf.
\]
For $s \in \R$ and $1 \le p, r \le \infty$, the Besov space $\Bb^s_{p,r}$ consists of all tempered distributions $u$ such that
\begin{equation}\label{def_Besov1}
\|u \|_{\Bb^s_{p,r}} :=
\left( \sum_{q=-1}^\infty \lambda_q^{rs} \| \Delta_q u \|_{L^p(\T^3)}^r \right)^{1/r}
< \infty,
\end{equation}
for $r < \infty$. In the case $r = \infty$, we define
\begin{equation}\label{def_Besov2}
\|u \|_{\Bb^s_{p,\infty}} :=
\sup_{q \ge -1} \lambda_q^{s} \| \Delta_q u \|_{L^p(\T^3)} < \infty.
\end{equation}
Bernstein's inequality implies the following embeddings:

\begin{Corollary}[Besov embeddings; see {\cite[Chapter~2]{bahouri2011fourier}}]\label{Corollary:Besov_embedding}
Let $q \geq p \geq 1$, $a>2$, $r\geq1$, and $s\geq0$. Then the following continuous embeddings hold:
\begin{align}
\Bb^s_{p,r} &\subset \Bb^{s - d\left(\frac{1}{p} - \frac{1}{q} \right)}_{q,r}, \label{Corollary:Besov_embedding_1} \\
\Bb^0_{a,2} &\subset L^{a}, \ \ \Bb^s_{2,2} = H^s. \label{Corollary:Besov_embedding_2}
\end{align}
\end{Corollary}

\section{Energy equality for the Nernst--Planck--Euler system}\label{sec:NP_Euler_Energy_equality_condition}
For smooth solutions $(u,\rho,\s)$ of \eqref{eq:NP_Euler}, one has the energy equality \eqref{eq:Energy_equality}. Indeed, testing \eqref{eq:NP_Euler_4} by $2 u$ and using \eqref{eq:NP_Euler_5} yields
\begin{align}
\frac{d}{dt} \twonorm{u(t)}^2 = -2 \int_{\T^3}  (\rho u) \cdot \nabla \Phi dx . \label{eq:Energy_equality_1}
\end{align}
On the other hand, testing \eqref{eq:NP_Euler_1} by $\Phi$ and using \eqref{eq:NP_Euler_3} gives
\begin{align*}
\frac{d}{dt} \int_{\T^3} \rho \Phi dx &= \int_{\T^3} \p_t \rho \Phi dx + \int_{\T^3} \rho \p_t \Phi dx
= \int_{\T^3} \p_t \rho \Phi dx + \int_{\T^3} (- \Delta \Phi) \p_t \Phi dx \\
&= \int_{\T^3} \p_t \rho \Phi dx + \int_{\T^3} \Phi \p_t \left( - \Delta \Phi \right)  dx
= 2 \int_{\T^3} \p_t \rho \Phi dx.
\end{align*}
Consequently,
\begin{align}
\frac{d}{dt} \int_{\T^3} \rho \Phi dx
&= -2 \int_{\T^3} \left( u \cdot \nabla \rho \right)\Phi dx + 2 \int_{\T^3} \Phi \left( \Delta \rho +\nabla \s \cdot \nabla \Phi + \s \Delta \Phi \right) dx \notag \\
&= 2 \int_{\T^3} \left( \rho u \right) \cdot \nabla \Phi dx  - 2   \int_{\T^3} |\Delta \Phi|^2 dx - 2 \int_{\T^3} \s |\nabla \Phi|^2 dx. \label{eq:Energy_equality_2}
\end{align}
Since
\begin{equation}\label{eq:Energy_equality_3}
\int_{\T^3} \rho \Phi dx = - \int_{\T^3} \Phi \Delta \Phi  dx =  \int_{\T^3} |\nabla \Phi |^2 dx,
\end{equation}
combining \eqref{eq:Energy_equality_1}--\eqref{eq:Energy_equality_3} and integrating in time yields \eqref{eq:Energy_equality}. We are interested in criteria that ensure the energy equality for weak solutions of \eqref{eq:NP_Euler}.
\begin{Definition}\label{def:weak_solution_NP_Euler}
We call $(u,\rho, \s)$ a weak solution to the system \eqref{eq:NP_Euler} on $[0,T]$ with weakly divergence-free initial data $u_0 \in L^2 (\T^3)$ and $\rho_0, \s_0 \in L^2 (\T^3)$, satisfying the compatibility condition $\intthree \rho_0\,dx=0$, if
\begin{align*}
u \in C_w ([0,T];L^2(\T^3)), \ \ \rho,\s \in C_w ([0,T];L^2(\T^3))
\end{align*}
and satisfy the following weak formulations:  
\begin{align}
&\Ltwoinner{u(t),\phi(t)} - \Ltwoinner{u_0,\phi(0)} - \int_0^t \Ltwoinner{u(s),\p_s \phi(s)} ds = \int_0^t \Ltwoinner{ u \cdot \nabla \phi , u }  ds -  \int_0^t \Ltwoinner{  \rho \nabla \Phi , \phi } ds, \label {def:weak_solution_NP_Euler_1} \\
&\Ltwoinner{\rho(t),\psi_1(t)} - \Ltwoinner{\rho_0,\psi_1(0)} - \int_0^t \Ltwoinner{\rho(s),\p_s \psi_1(s) + u \cdot \nabla \psi_1 + \Delta \psi_1} ds  = -  \int_0^t \Ltwoinner{\s \nabla \Phi, \nabla \psi_1}ds, \label {def:weak_solution_NP_Euler_2} \\
&\Ltwoinner{\s(t),\psi_2(t)} - \Ltwoinner{\s_0,\psi_2(0)} - \int_0^t \Ltwoinner{\s(s),\p_s \psi_2(s) + u \cdot \nabla \psi_2 + \Delta \psi_2} ds  = -  \int_0^t \Ltwoinner{\rho \nabla \Phi, \nabla \psi_2}ds, \label {def:weak_solution_NP_Euler_3}
\end{align}
\begin{equation}\label{def:weak_solution_NP_Euler_4}
 \Ltwoinner{\rho +  \Delta \Phi, \psi_3 } = 0.
\end{equation}
The potential is normalized by $\intthree \Phi(t,x)\,dx=0$, and $\rho$ satisfies $\intthree \rho(t,x)\,dx=0$.  These identities hold for every $t \in [0,T]$ and every $\phi, \psi_1, \psi_2, \psi_3 \in C^\infty ([0,T] \times \T^3)$ such that $\phi$ is divergence-free. In addition, $\nabla\cdot u=0$ in $\mathcal D'((0,T)\times\T^3)$.
\end{Definition}

\begin{Remark}
For the purpose of proving the energy equality, the condition
$\rho \in L^2(\mathbb T^3)$ is sufficient, by
\eqref{def:weak_solution_NP_Euler_4}. Moreover, since the weak formulation
above extends to test functions that are smooth in space and weakly
Lipschitz continuous in time, one may use spatially mollified weak
solutions as admissible test functions in
\eqref{def:weak_solution_NP_Euler_1}--\eqref{def:weak_solution_NP_Euler_4}, after first applying a standard time regularization and then passing to the limit.
\end{Remark}

For physical reasons, it is also important that weak solutions preserve non-negativity of the ion concentrations throughout the evolution. We therefore introduce the following definition of a parabolic weak solution. The terminology originates from the study of weak solutions to advection--diffusion equations; see \cite{BCC24}.

\begin{Definition}\label{def:parabolic_solution}
A weak solution of the system \eqref{eq:NP_Euler} on $[0,T] \times \T^3$ is called \textit{parabolic} if
\[
\rho,\s \in L^2([0,T];H^1(\T^3)).
\]
\end{Definition}

For a smooth velocity field $u$, the following lemma shows that parabolic solutions preserve non-negativity provided that the initial ion concentrations are non-negative. To this end, we introduce the following quantities:
\begin{equation}\label{def:ion_concentration}
\Theta(t) := \haf \left( \s(t) + \rho(t) \right), \ \ \ \theta(t) := \haf \left( \s(t) - \rho(t) \right),
\end{equation}
which represent the concentrations of the two ionic species, respectively. Combining the weak formulations \eqref{def:weak_solution_NP_Euler_2} and \eqref{def:weak_solution_NP_Euler_3} yields the following identities in the sense of distributions:
\begin{align}
\p_t \th + u \cdot \grad \theta &= \Delta \theta - \divv \left( \theta \nabla \Phi \right) \label{lemma:non-negativity_smooth_NPEuler_1} \\
\p_t \Theta + u \cdot \grad \Theta &= \Delta \Theta + \divv \left( \Theta \nabla \Phi \right) \label{lemma:non-negativity_smooth_NPEuler_2}
\end{align}

\begin{Lemma}\label{lemma:non-negativity_smooth_NPEuler}
Let $(u,\rho,\s)$ be a parabolic solution with $u$ smooth. Then $\Theta(t), \theta(t) \ge 0$ for all $t \in [0,T]$ whenever $\Theta_0, \theta_0 \ge 0$.
\end{Lemma}

\begin{proof}
We will show that $\theta$ remains non-negative provided that $\theta_0 \ge 0$. The argument for $\Theta$ follows analogously. Denote
\[
\th^- (t) := \max \{-\theta(t),0 \}.
\]
Since $\rho,\s \in L^2([0,T];H^1(\T^3))$, it follows that $\theta$ also belongs to $L^2([0,T];H^1(\T^3))$. We claim that $\p_t \th \in L^2(0,T;H^{-1})$. Indeed, for $\phi \in H^1(\T^3)$,
\begin{align*}
\Ltwoinner{\p_t \th, \phi} &= \Ltwoinner{- u \cdot \grad \th, \phi} + \Ltwoinner{\Delta\th, \phi} + \Ltwoinner{-\divv (\th \grad \Phi), \phi} \\
&\leq \threenorm{u}\twonorm{\grad \th} \sixNorm{\phi} + \twonorm{\grad \th} \Honenorm{\phi} + \threenorm{\th} \sixNorm{\grad \Phi} \Honenorm{\phi} \\
&\les \left( \threenorm{u}  + 1 + \twonorm{\rho} \right) \Honenorm{\th} \Honenorm{\phi},
\end{align*}
and therefore
\[
\|\p_t \th \|_{H^{-1}(\T^3)} \les \left( \threenorm{u}  + 1 + \twonorm{\rho} \right) \Honenorm{\th}.
\]
Therefore, $\p_t \th \in L^2(0,T;H^{-1})$ as $u$ is smooth and $\rho \in L^\infty([0,T];L^2(\T^3))$.
By a density argument, we may test equation \eqref{lemma:non-negativity_smooth_NPEuler_1} with $-\theta^{-}$ to obtain
\begin{align}
    \haf \twonorm{\th^-(t)}^2 - \haf \twonorm{\th^-_0}^2 + \int_0^t \Ltwoinner{u \cdot \grad \th, \th^-}  ds &= -\int_0^t \twonorm{\grad \th^-}^2 ds - \int_0^t  \Ltwoinner{\th \grad \Phi, \grad \th^-}  ds, \notag
\end{align}
and therefore,
\begin{align}
     \haf \twonorm{\th^-(t)}^2 - \haf \twonorm{\th^-_0}^2  &= -\int_0^t \twonorm{\grad \th^-}^2 ds + \frac12\int_0^t  \Ltwoinner{\rho, |\th^-|^2}  ds  \label{lemma:non-negativity_smooth_NPEuler_3}
\end{align}
where we used integration by parts, the identity $\th \th^- = -|\th^-|^2$, and the divergence-free condition on $u$. Using H\"older's and Ladyzhenskaya inequalities, we estimate the last term of \eqref{lemma:non-negativity_smooth_NPEuler_3} and get
\begin{align*}
&\haf \twonorm{\th^-(t)}^2 - \haf \twonorm{\th^-_0}^2  \leq -\int_0^t \twonorm{\grad \th^-}^2 ds + \frac12\|\rho \|_{L^\infty_t L^2_x} \int_0^t  \fournorm{\th^-}^2 ds \\
&\leq - \int_0^t \twonorm{\grad \th^-}^2 ds +C \|\rho \|_{L^\infty_t L^2_x} \int_0^t  \twonorm{\th^-}^\haf (\twonorm{\grad \th^-}^\frac{3}{2} + \twonorm{\th^-}^\frac{3}{2}) ds \\
& \ \leq - \int_0^t \twonorm{\grad \th^-}^2 ds + C (1+\|\rho \|_{L^\infty_t L^2_x}^4) \int_0^t  \twonorm{\th^-}^2 ds + \int_0^t \twonorm{\grad \th^-}^2 ds \\
&\ \ \leq C (1+\|\rho \|^4_{L^\infty_t L^2_x}) \int_0^t  \twonorm{\th^-}^2 ds.
\end{align*}
As $\theta_0^- = 0$, we conclude from the above inequality and Gr\"onwall's lemma that $\theta^-(t) = 0$ for all $t \in [0,T]$.
\end{proof}

We now give an Onsager-type regularity condition ensuring that weak solutions of \eqref{eq:NP_Euler} satisfy the energy equality. In addition, we establish a sufficient condition under which the ion concentrations remain non-negative under only Onsager-type regularity.
\begin{Theorem}\label{Thm:Energy_equality_NPEuler}
Let $(u,\rho,\s)$ be a weak solution to the equations \eqref{eq:NP_Euler} on $[0,T]$ such that
\begin{align*}
u \in L^3( [0,T];\Bb^{\frac{1}{3}}_{3,\infty}(\T^3))
\end{align*}
and suppose that the vanishing dyadic flux condition
\begin{equation}\label{Thm:Energy_equality_NPEuler_1}
\lim_{q \to \infty} \int_0^T \l_q \threenorm{\Dq u (t)}^3 dt  = 0
\end{equation}
holds. Then \eqref{eq:Energy_equality} holds for $(u,\rho,\s)$.
\end{Theorem}

\noindent Theorem \ref{Thm:Energy_equality_NPEuler} requires no additional assumptions on $\rho$ and $\s$ beyond Definition~\ref{def:weak_solution_NP_Euler} for the energy equality \eqref{eq:Energy_equality} to hold. To preserve non-negativity along the evolution, it suffices to assume the following.
\begin{Theorem}\label{Thm:non-negativity_NPEuler}
Suppose that a weak solution $(u,\rho,\s)$ to the system \eqref{eq:NP_Euler} satisfies the conditions in Theorem \ref{Thm:Energy_equality_NPEuler}. If $\rho, \s \in L^3([0,T];\Bb^{\frac{1}{3}}_{3,\infty}(\T^3))$, then $\th(t)$ and $\Theta(t)$ defined by \eqref{def:ion_concentration} remain non-negative whenever $\th_0$ and $\Theta_0$ are non-negative. Furthermore, $(u,\rho,\s)$ is a parabolic solution and satisfies the energy equality
\begin{align}
\haf \left( \twonorm{\rho(t)}^2 + \twonorm{\s(t)}^2 \right) - \haf \left( \twonorm{\rho_0}^2 + \twonorm{\s_0}^2 \right) \notag \\
= -\int_0^t \left( \twonorm{\grad \rho}^2 + \twonorm{\grad \s}^2 \right) ds - \int_0^t  \Ltwoinner{\s, |\rho|^2}  ds. \label{Thm:non-negativity_NPEuler_1}
\end{align}
\end{Theorem}


\subsection{Proof of Theorem \ref{Thm:Energy_equality_NPEuler}}
The proof is an adaptation of the Constantin--E--Titi commutator method as used in \cite{CCFS08}. We define the fluid energy flux $\Pi^u_q$ at the $q$-th dyadic shell by
\begin{equation}\label{def:Flow_energy_flux}
\Pi^u_q := \intthree \tr \left [ \Ss_q \left( u \otimes u \right) \cdot \nabla \Ss_q u \right] dx.
\end{equation}
Using $\Ss_q^2 u$ as a test function in the weak formulation of the Euler equation \eqref{def:weak_solution_NP_Euler_1}, and exploiting the self-adjointness of $\Ss_q$, we obtain
\begin{align*}
\haf \twonorm{\Ss_q u}^2 - \haf \twonorm{\Ssq u_0}^2 +  \int_0^t  \intthree \Ssq \div \left( u \otimes u \right) \cdot \Ss_q u \, dx ds &= - \int_0^t \intthree \Ssq(\rho \nabla \Phi) \cdot \Ss_q u \, dx ds
\end{align*}
\begin{equation}\label{proof_Thm_energy_equality_1}
\Rightarrow \ \haf  \twonorm{\Ss_q u}^2 - \haf \twonorm{\Ssq u_0}^2  = \int_0^t \Pi^u_q \,ds - \int_0^t \intthree \Ss_q \left( \rho \nabla \Phi \right) \cdot \Ss_q u \,dx ds.
\end{equation}
Similarly, testing the weak formulation \eqref{def:weak_solution_NP_Euler_2} with $\Ss^2_q \Phi$ yields
\begin{align}
\haf  \twonorm{\Ss_q \nabla \Phi}^2 - \haf \twonorm{\Ssq \grad \Phi_0}^2 &= \int_0^t \intthree \Ss_q (\rho u) \cdot \Ss_q \nabla \Phi \, dx ds - \int_0^t \intthree | \Ss_q \Delta \Phi |^2 \,dx ds \notag \\
&\ - \int_0^t \intthree \Ss_q (\s \nabla \Phi) \cdot \Ss_q \nabla \Phi \,dx ds.\label{proof_Thm_energy_equality_2}
\end{align}
Define the following ionic flux terms at the $q$-th dyadic scale by
\begin{align}
\Pi^1_q &:= \intthree \bigg( \Ss_q (\rho u) \cdot \Ss_q \nabla \Phi - \Ss_q (\rho \nabla \Phi) \cdot \Ss_q u \bigg) dx, \label{def:ion_energy_flux_1} \\
\Pi^2_q &:= \intthree \left[ \Big( \Ss_q \s \Ss_q \nabla \Phi  - \Ss_q (\s \nabla \Phi) \Big) \cdot \Ss_q \nabla \Phi \right] dx, \label{def:ion_energy_flux_2}
\end{align}
then adding \eqref{proof_Thm_energy_equality_1} and \eqref{proof_Thm_energy_equality_2} leads to
\begin{align}
  & \twonorm{\Ss_q u(t)}^2 - \twonorm{\Ss_q u_0}^2  +  \twonorm{\Ss_q \nabla \Phi(t)}^2 -  \twonorm{\Ss_q \nabla \Phi_0}^2  \notag \\
=\ &2 \int_0^t \Big( \Pi^u_q (s) + \Pi^1_q (s) +  \Pi^2_q (s) \Big) ds
- 2  \int_0^t \intthree |\Ssq \Delta \Phi|^2 dx ds - 2 \int_0^t \intthree \Ssq \s |\Ssq \nabla \Phi|^2 dx ds, \label{proof_Thm_energy_equality_3}
\end{align}
after integrating over time.

\noindent Hence, to establish the energy equality \eqref{eq:Energy_equality} it suffices to show that
\begin{equation}
\lim_{q \to \infty} \int_0^t \Big( |\Pi^u_q | + |\Pi^1_q | + |\Pi^2_q | \Big) ds  = 0. \label{proof_Thm_energy_equality_4}
\end{equation}
We also record the limiting properties of the non-flux terms in \eqref{proof_Thm_energy_equality_3}. For each fixed time, $\Ss_q u\to u$, $\Ss_q\rho\to\rho$, and $\Ss_q\s\to\s$ strongly in $L^2(\T^3)$. Moreover, from $-\Delta\Phi=\rho$ and the zero-mean normalization of $\Phi$, we have $\Ss_q\nabla\Phi\to\nabla\Phi$ strongly in $H^1(\T^3)$ and $\Ss_q\Delta\Phi\to\Delta\Phi$ strongly in $L^2(\T^3)$. Thus
\[
\int_0^t\twonorm{\Ss_q\Delta\Phi}^2ds\to \int_0^t\twonorm{\Delta\Phi}^2ds.
\]
For the drift dissipation, elliptic regularity gives $\nabla\Phi\in L^\infty(0,T;H^1(\T^3))\hookrightarrow L^\infty(0,T;L^6(\T^3))$, and hence $|\Ss_q\nabla\Phi|^2\to |\nabla\Phi|^2$ strongly in $L^r(0,t;L^3(\T^3))$ for every finite $r$. Together with $\Ss_q\s\to\s$ strongly in $L^{r'}(0,t;L^2(\T^3))$, H\"older's inequality yields
\[
\int_0^t\intthree \Ss_q\s |\Ss_q\nabla\Phi|^2dxds\to
\int_0^t\intthree \s |\nabla\Phi|^2dxds.
\]
Therefore it remains to prove the zero-flux condition \eqref{proof_Thm_energy_equality_4}.

We first consider the flux term $\Pi^u_q$. Note that one can write
\begin{equation}\label{proof_Thm_energy_equality_5}
\Ssq (u \otimes u) = \Ssq u \otimes \Ssq u + r_q (u,u) - \left( u - \Ssq u \right) \otimes \left( u - \Ssq u \right),
\end{equation}
where
\[
r_{q} (u,u) = \l^3_{q+1} \intthree \widetilde{h}(\l_{q+1} y) \left( u(x-y) - u(x) \right) \otimes \left( u(x-y) - u(x) \right) dy,
\]
which follows from \eqref{eq:q_mode_cutoff}. Then substituting \eqref{proof_Thm_energy_equality_5} into \eqref{def:Flow_energy_flux} yields
\begin{align*}
    \Pi^u_q &= \intthree \left[ \tr \Big( r_q (u,u) \cdot \nabla \Ss_q u \Big) \right] dx - \intthree \left[ \tr \Big( \left( u - \Ssq u \right) \otimes \left( u - \Ssq u \right) \cdot \nabla \Ss_q u \Big) \right] dx \\
    &:= A_1 + A_2
 \end{align*}
The term $A_1$ can be bounded by H\"older's inequality via
\[
A_1 = \intthree \left[ \tr \Big( r_q (u,u) \cdot \nabla \Ss_q u \Big) \right] dx \leq \threetwoNorm{r_q(u,u)} \threeNorm{\nabla \Ssq u}.
\]
Note that
\[
\threetwoNorm{r_q(u,u)} \leq \l^3_{q+1} \intthree \left| \widetilde{h}(\l_{q+1} y) \right| \threenorm{u(\cdot - y) - u(\cdot)}^2 dy.
\]
In view of Lemma \ref{lemma:Bernstein} and Corollary \ref{Corollary:Besov_embedding}, we can estimate as
\begin{align*}
\threenorm{u(\cdot - y) - u(\cdot)}^2 &\leq \threeNorm{\Ssq \left( u(\cdot - y) - u(\cdot) \right)}^2 + \threeNorm{\left(\IdM - \Ssq \right) \left( u(\cdot - y) - u(\cdot) \right)}^2 \\
&\les \sum_{j = -1}^q |y|^2 \l^2_j \threenorm{\Dj u}^2 + \sumjgtrq \threenorm{\Dj u}^2 \\
&\les \l_q^{\frac{4}{3}} \sumjtq |y|^2 \l^{-\frac{4}{3}}_{q-j} \left ( \l^{\onethird}_j \threenorm{\Dj u} \right)^2 + \l_q^{-\frac{2}{3}} \sumjgtrq \l^{\frac{2}{3}}_{q-j} \left( \l^{\frac{1}{3}}_j \threenorm{\Dj u} \right)^2.
\end{align*}
Denote the frequency localization kernel by
\begin{align}
K(q) &= \begin{cases} \label{def:frequency_localization_kernel}
\l^{\frac{2}{3}}_q, \ \ q \leq 0, \\\\
\l^{-\frac{4}{3}}_q, \ \ q > 0,
\end{cases}
\end{align}
and, for the velocity field,
\begin{align}
d_j(s) = \l^{\onethird}_j \threenorm{ \Delta_j u(s)}, \qquad j\geq -1,
\quad \text{and} \quad d_j(s)=0 \ \text{for } j<-1. \label{def:d}
\end{align}
We write
\[
(K*d^m)(q,s):=\sum_{j\geq -1}K(q-j)d_j(s)^m,
\]
and suppress the time-dependence when no confusion is possible. Then the bound of the above terms can be written as
\begin{equation}\label{proof_Thm_energy_equality_6}
\threenorm{u(\cdot - y) - u(\cdot)}^2 \les \left ( \l^{\frac{4}{3}}_q |y|^2 + \l_q^{-\frac{2}{3}} \right) \Big( K \ast d^2 \Big)(q)
\end{equation}
and by Corollary \ref{Corollary:Besov_embedding},
\begin{align*}
    A_1 &\les \Big( K \ast d^2 \Big)(q) \left( \l^3_{q+1} \intthree \l^{\frac{4}{3}}_q |y|^2 \widetilde{h}(\l_{q+1} y) dy + \l_q^{-\frac{2}{3}}  \right) \left( \sumjtq \l^2_j \threenorm{\Dj u}^2 \right)^\haf \\
    &\les \Big( K \ast d^2 \Big)(q) \l^{-\frac{2}{3}}_q \left( \sumjtq \l^2_j \threenorm{\Dj u}^2 \right)^\haf \\
    &\les \Big( K \ast d^2 \Big)^{\frac{3}{2}}(q).
\end{align*}
Similarly for $A_2$ we have that
\begin{align*}
    A_2 &\leq \threenorm{u - \Ssq u}^2 \threenorm{\nabla \Ssq u} \\
    &\les \left( \sumjgtrq \threenorm{\Dj u}^2 \right) \left( \sumjtq \l^2_j \threenorm{\Dj u}^2 \right)^\haf \\
    &\les \l^{-\frac{2}{3}}_q \left( \sumjgtrq \l_{q-j}^{\frac{2}{3}} \left( \l^{\frac{1}{3}}_j \threenorm{\Dj u} \right)^2 \right)  \left( \sumjtq \l^2_j \threenorm{\Dj u}^2 \right)^\haf \\
    &\les \Big( K \ast d^2 \Big)(q) \left( \sumjtq \l^{-\frac{4}{3}}_{q-j} \left(\l^{\onethird}_j \threenorm{\Dj u} \right)^2 \right)^\haf  = \Big( K \ast d^2 \Big)^{\frac{3}{2}}(q).
\end{align*}
These imply that
\[
|\Pi^u_q| \les \Big( K \ast d^2 \Big)^{\frac{3}{2}}(q)
\]
and furthermore
\begin{align*}
    \int_0^t |\Pi^u_q(s)| ds &\les \int_0^t \Big( K \ast d^2 \Big)^{\frac{3}{2}}(q,s) ds.
\end{align*}
We now keep the convolution structure instead of reducing it to a single dyadic shell. Let
$M_q:=\sum_{j\geq -1}K(q-j)$. Since $K\in\ell^1(\mathbb Z)$, $M_q\leq \|K\|_{\ell^1}$. By Jensen's inequality applied to the convex function $x\mapsto x^{3/2}$,
\begin{align*}
\Big( K \ast d^2 \Big)^{\frac{3}{2}}(q,s)
&=M_q^{\frac32}\left(\sum_{j\geq -1}\frac{K(q-j)}{M_q}d_j^2(s)\right)^{\frac32}\\
&\leq M_q^{\frac32}\sum_{j\geq -1}\frac{K(q-j)}{M_q}d_j^3(s)\\
&\leq \|K\|_{\ell^1}^{\frac12}\sum_{j\geq -1}K(q-j)d_j^3(s)
\les (K*d^3)(q,s).
\end{align*}
Therefore,
\begin{align}
\int_0^t |\Pi^u_q(s)| ds
&\les \sum_{j\geq -1}K(q-j)\int_0^t \lambda_j\|\Delta_j u(s)\|_{L^3(\mathbb T^3)}^3ds. \label{proof_Thm_energy_equality_7}
\end{align}
By \eqref{Thm:Energy_equality_NPEuler_1}, the sequence
$a_j:=\int_0^t \lambda_j\|\Delta_j u(s)\|_{L^3(\mathbb T^3)}^3ds$ tends to zero as $j\to\infty$. Since $K\in\ell^1(\mathbb Z)$, the right-hand side of \eqref{proof_Thm_energy_equality_7} is an $\ell^1$-convolution of $K$ with a sequence in $c_0$, and hence tends to zero as $q\to\infty$. Thus
\[
\lim_{q\to\infty}\int_0^t |\Pi^u_q(s)|ds=0.
\]

\noindent The fluxes $\Pi^1_q$ and $\Pi^2_q$ are handled with a similar idea. For the former, we write
\begin{align*}
\Ssq \left( \rho u \right) = \Ssq \rho \Ssq u + r_q (\rho, u) - \left(\rho-\Ssq \rho \right)\left(u-\Ssq u \right),
\end{align*}
where
\[
r_q(\rho,u) = \l^3_{q+1} \intthree \widetilde{h}(\l_{q+1} y) \left( \rho(x-y) - \rho(x) \right)\left( u(x-y) - u(x) \right) dy.
\]
In addition, we have
\[
\Ssq (\rho \nabla \Phi) = \Ssq \rho \Ssq \nabla \Phi + r_q (\rho,\nabla \Phi) - \left( \rho - \Ssq \rho \right) \left( \nabla \Phi - \Ssq \nabla \Phi \right),
\]
such that
\[
r_q(\rho,\nabla \Phi) = \l^3_{q+1} \intthree \widetilde{h}(\l_{q+1} y) \left( \rho(x-y) - \rho(x) \right)\left( \nabla \Phi(x-y) - \nabla \Phi(x) \right) dy.
\]
Therefore the integrand in $\Pi^1_q$ satisfies
\begin{align*}
 \Ssq \left( \rho u \right) \cdot \Ssq \nabla \Phi -  \Ssq (\rho \nabla \Phi) \cdot \Ssq u &=  r_q (\rho,u) \cdot \Ssq \nabla \Phi - r_q(\rho, \nabla \Phi) \cdot \Ssq u \\
 &- \left(\rho-\Ssq \rho \right)\left(u-\Ssq u \right) \cdot \Ssq \nabla \Phi + \left( \rho - \Ssq \rho \right) \left( \nabla \Phi - \Ssq \nabla \Phi \right) \cdot \Ssq u.
\end{align*}
We estimate the four terms on the right-hand side separately. By Sobolev embedding and the estimate \eqref{proof_Thm_energy_equality_6} we obtain that
\begin{align}
\intthree \left| r_q (\rho,u) \cdot \Ssq \nabla \Phi \right| dx &\leq \sixfiveNorm{r_q(\rho,u)} \sixNorm{\Ssq \nabla \Phi} \lesssim\sixfiveNorm{r_q(\rho,u)} \twonorm{\rho}, \notag \\
& \lesssim \twonorm{\rho} \intthree \l_{q+1}^3 \widetilde{h} (\l_{q+1} y) \twonorm{\rho(\cdot - y) - \rho(\cdot)} \threenorm{u(\cdot - y) - u(\cdot)} dy, \notag \\
&\lesssim \twonorm{\rho}^2 \intthree \l_{q+1}^3 \widetilde{h} (\l_{q+1} y) \threenorm{u(\cdot - y) - u(\cdot)} dy, \notag \\
& \les \twonorm{\rho}^2 \intthree \l_{q+1}^3 \widetilde{h} (\l_{q+1} y) \left( \l_q^{\frac{2}{3}} |y| +\l_q^{-\onethird} \right) \Big( K \ast d^2 \Big)^\haf (q) dy, \notag \\
&\les \twonorm{\rho}^2 \l^{-\onethird}_q \Big( K \ast d^2 \Big)^\haf (q). \label{proof_Thm_energy_equality_8}
\end{align}
For the term containing $r_q(\rho,\nabla\Phi)$ we use the elliptic relation $-\Delta\Phi=\rho$ and the low-frequency Bernstein estimate
\[
\Linfnorm{\Ssq u}
\leq \sum_{j=-1}^{q}\Linfnorm{\Delta_j u}
\les \sum_{j=-1}^{q}\lambda_j\threenorm{\Delta_j u}
\les \lambda_q^{\frac23}\|u\|_{\Bb^{\frac13}_{3,\infty}}
\sum_{j=-1}^{q}\lambda_{j-q}^{\frac23}
\les \lambda_q^{\frac23}\|u\|_{\Bb^{\frac13}_{3,\infty}}.
\]
Moreover,
\[
\onenorm{r_q(\rho,\nabla\Phi)}
\les \twonorm{\rho}\,\twonorm{\nabla^2\Phi}\intthree |y|\lambda_{q+1}^3|\widetilde h(\lambda_{q+1}y)|dy
\les \lambda_q^{-1}\twonorm{\rho}^2.
\]
Therefore
\begin{align}
    \intthree \left|r_q(\rho,\nabla \Phi) \cdot \Ssq u \right| dx
    &\leq \onenorm{r_q(\rho,\nabla \Phi)} \LinfNorm{\Ssq u} \notag \\
    &\les \lambda_q^{-\frac13}\twonorm{\rho}^2\|u\|_{\Bb^{\frac13}_{3,\infty}}. \label{proof_Thm_energy_equality_9}
\end{align}
Next, we bound the remaining terms. First,
\begin{align}
\intthree \left | \left( \rho - \Ssq \rho \right) (u - \Ssq u) \cdot \Ssq \nabla \Phi \right| dx &\leq \twonorm{\rho - \Ssq \rho} \threenorm{u - \Ssq u} \sixNorm{\Ssq \nabla \Phi} \notag \\
&\les \lambda_q^{-\onethird}\twonorm{\rho}^2 \Big( K \ast d^2 \Big)^\haf(q,s). \label{proof_Thm_energy_equality_10}
\end{align}
For the last term, since $-\Delta\Phi=\rho$,
\[
\twonorm{\grad\Phi-\Ssq\grad\Phi}
\les \lambda_q^{-1}\twonorm{\rho}.
\]
Together with the preceding bound for $\Linfnorm{\Ssq u}$, this gives
\begin{align}
&\intthree \left | \left( \rho - \Ssq \rho \right) (\grad \Phi - \Ssq \grad \Phi ) \cdot \Ssq u \right| dx \notag \\
&\qquad \leq \twonorm{\rho - \Ssq \rho} \twonorm{\grad \Phi - \Ssq \grad \Phi} \Linfnorm{\Ssq u} \notag \\
&\qquad \les \lambda_q^{-\frac13}\twonorm{\rho}^2\|u\|_{\Bb^{\frac13}_{3,\infty}}. \label{proof_Thm_energy_equality_11}
\end{align}
Collecting the estimates \eqref{proof_Thm_energy_equality_8}-\eqref{proof_Thm_energy_equality_11}, and using Jensen's inequality in the form
$(K*d^2)^{1/2}(q,s)\les (K*d^3)^{1/3}(q,s)$, we obtain
\begin{align}
\int_0^t |\Pi^1_q(s)|ds
&\les \|\rho\|^2_{L^{\infty}([0,T];L^2)}\lambda_q^{-\frac13}
\int_0^t \left[(K*d^2)^\haf(q,s)+\|u(s)\|_{\Bb^{\frac13}_{3,\infty}}\right]ds \notag\\
&\les \|\rho\|^2_{L^{\infty}([0,T];L^2)}\lambda_q^{-\frac13}
\left(
T^{\frac23}\left(\int_0^t (K*d^3)(q,s)ds\right)^{\frac13}
+\|u\|_{L^1(0,T;\Bb^{\frac13}_{3,\infty})}
\right). \label{proof_Thm_energy_equality_12}
\end{align}
The first factor in parentheses is bounded, and in fact tends to zero by \eqref{Thm:Energy_equality_NPEuler_1} and $K\in\ell^1(\mathbb Z)$, while the second factor is finite. Hence the explicit factor $\lambda_q^{-1/3}$ implies
\[
\lim_{q\to\infty}\int_0^t |\Pi^1_q(s)|ds=0.
\]

\noindent Finally, repeating the preceding argument gives the bound for $\Pi^2_q$. Note that the integrand of the flux becomes
\[
\Big( \Ssq \s \Ssq \grad \Phi - \Ssq (\s \grad \Phi) \Big) \cdot \Ssq \grad \Phi = \Big (- r_q (\s,\grad \Phi) +(\s - \Ss_q \s) (\grad \Phi - \Ssq \grad \Phi) \Big) \cdot \Ssq \grad \Phi,
\]
where
\[
r_q (\s,\grad \Phi) = \l^3_{q+1} \intthree \widetilde{h}(\l_{q+1} y) \left( \s(x-y) - \s(x) \right)\left( \nabla \Phi(x-y) - \nabla \Phi(x) \right) dy.
\]
We use the following estimate
\begin{align*}
\threenorm{\grad \Phi (\cdot - y) - \grad \Phi (\cdot)}^2 &\les \sumjtq |y|^2 \l_j^4 \threenorm{\Dj \Phi}^2 + \sumjgtrq \l_j^2 \threenorm{\Dj \Phi}^2 \\
&\les \sumjtq |y|^2 \l_j^5 \twonorm{\Dj \Phi}^2 + \sumjgtrq \l_j^3 \twonorm{\Dj \Phi}^2 \\
&\les |y|^2 \l_q \Honenorm{\grad \Phi}^2 + \l_q^\mone \Honenorm{\grad \Phi}^2,
\end{align*}
where we applied Lemma \ref{lemma:Bernstein}. It then holds that
\begin{align}
&\intthree \left| r_q (\s,\grad \Phi) \cdot \Ssq \grad \Phi \right| dx \leq \sixfiveNorm{r_q (\s, \grad \Phi)} \sixNorm{\Ssq \grad \Phi} \notag \\
\leq\ \ &\twonorm{\rho} \intthree \l_{q+1}^3 \wth (\l_{q+1} y) \twonorm{\s(\cdot-y) - \s(\cdot)} \threenorm{\nabla \Phi (\cdot - y) - \grad \Phi(\cdot)} dy \notag \\
\les\ \ &\l^{-\haf}_q \twonorm{\rho} \twonorm {\s} \Honenorm{\grad \Phi} \notag \\
\les\ \ &\l^{-\haf}_q \twonorm{\rho}^2 \twonorm {\s}. \label{proof_Thm_energy_equality_13}
\end{align}
Hence, 
\begin{align*}
    \intthree \l_{q+1}^3 \wth(\l_{q+1} y) \threenorm{\grad \Phi (\cdot -y) -\grad \Phi(\cdot)} dy &\les \Honenorm{\grad \Phi} \intthree \l^\haf_q |y| \l^3_{q+1} \wth(\l_{q+1} y) dy  \\
    &+     \Honenorm{\grad \Phi} \intthree \l^{-\haf}_q  \l^3_{q+1} \wth(\l_{q+1} y) dy    \\
    &\les \l^{-\haf}_q \Honenorm{\grad \Phi}.
\end{align*}
On the other hand, using Bernstein's inequality again, we have
\begin{align}
\intthree \left | \left( \s - \Ssq \s \right) (\grad \Phi - \Ssq \grad \Phi ) \cdot \Ssq \grad \Phi \right| dx &\leq \twonorm{\s - \Ssq \s} \threenorm{\grad \Phi - \Ssq \grad \Phi} \sixNorm{\Ssq \grad \Phi} \notag \\
&\les \twonorm{\s}  \left( \sumjgtrq \l^2_j \threenorm{\Dj \Phi}^2 \right)^\haf \Honenorm{\grad \Phi} \notag \\
&\les \l^{-\haf}_q \twonorm{\rho} \twonorm{\s}  \left( \sumjgtrq \l^4_j \twonorm{\Dj \Phi}^2 \right)^\haf  \notag \\
&\les  \l_q^{-\haf} \twonorm{\rho}^2 \twonorm{\s} \label{proof_Thm_energy_equality_14}.
\end{align}
The bounds \eqref{proof_Thm_energy_equality_13} and \eqref{proof_Thm_energy_equality_14} vanish uniformly as $q$ tends to infinity. Thus it follows directly that
\begin{equation}\label{proof_Thm_energy_equality_15}
\lim_{q \to \infty} \int_0^t |\Pi^2_q| ds  = 0.
\end{equation}
Combining \eqref{proof_Thm_energy_equality_7}, \eqref{proof_Thm_energy_equality_12}, and \eqref{proof_Thm_energy_equality_15}, we establish the zero-flux condition \eqref{proof_Thm_energy_equality_4} and conclude the proof.    \qed

\subsection{Proof of Theorem \ref{Thm:non-negativity_NPEuler}} \label{proof:Thm:non-negativity_NPEuler}
We first prove the parabolic regularity and the energy identity for $\th$. Testing \eqref{lemma:non-negativity_smooth_NPEuler_1} with $\Ss_q^2\th$, using the self-adjointness of $\Ss_q$, and using that $\Ss_q$ commutes with derivatives, we obtain
\begin{align}
    &\haf \twonorm{\Ss_q \th (t)}^2 - \haf \twonorm{\Ss_q \th_0}^2 +  \int_0^t \twonorm{\Ss_q \nabla \th(s)}^2 ds \notag \\
    &= \int_0^t \Ltwoinner{\Ssq (\th u), \Ssq \grad \th} ds + \int_0^t \Ltwoinner{\Ssq (\th \grad \Phi) - \Ssq \th \Ssq \grad \Phi, \Ssq \grad \th} ds  + \haf \int_0^t \Ltwoinner{|\Ssq \th|^2, \Ssq \rho} ds. \label{proof:Thm_non-negativity_NPEuler_1}
\end{align}
Indeed, the last term follows from the elliptic relation $-\Delta\Phi=\rho$:
\[
\Ltwoinner{\Ssq \th \Ssq \grad\Phi,\Ssq\grad\th}
=\haf\intthree \Ssq\grad\Phi\cdot\grad |\Ssq\th|^2\,dx
=-\haf\intthree \Ssq\Delta\Phi |\Ssq\th|^2\,dx
=\haf\Ltwoinner{|\Ssq\th|^2,\Ssq\rho}.
\]
We next show that the first two terms on the right-hand side of \eqref{proof:Thm_non-negativity_NPEuler_1} vanish as $q\to\infty$. For $f\in\{u,\th,\Theta,\rho\}$, set
\[
d_{f,j}(s)=\lambda_j^{\frac13}\threenorm{\Delta_j f(s)},\qquad j\ge -1,
\]
and set $d_{f,j}=0$ for $j<-1$. We also write
\[
(K*d_f^m)(q,s):=\sum_{j\ge -1}K(q-j)d_{f,j}(s)^m,
\]
where $K$ is the kernel defined in \eqref{def:frequency_localization_kernel}. Since $K\in\ell^1(\mathbb Z)$, Jensen's inequality gives
\begin{equation}\label{proof:Thm_non-negativity_NPEuler_Jensen}
(K*d_f^2)^{\frac12}(q,s)\les (K*d_f^3)^{\frac13}(q,s),
\qquad
(K*d_f^2)(q,s)\les (K*d_f^3)^{\frac23}(q,s).
\end{equation}
The standard translation estimates used above give
\begin{equation}\label{proof:Thm_non-negativity_NPEuler_translation}
\threenorm{f(\cdot-y)-f(\cdot)}
\les \left(\lambda_q^{\frac23}|y|+\lambda_q^{-\frac13}\right)(K*d_f^2)^{\frac12}(q,s),
\end{equation}
and, in particular,
\begin{equation}\label{proof:Thm_non-negativity_NPEuler_cutoff}
\threenorm{\Ssq\grad f}\les \lambda_q^{\frac23}(K*d_f^2)^{\frac12}(q,s),
\qquad
\threenorm{f-\Ssq f}\les \lambda_q^{-\frac13}(K*d_f^2)^{\frac12}(q,s).
\end{equation}

For the transport contribution, we use the Constantin--E--Titi product identity
\begin{equation}\label{proof:Thm_non-negativity_NPEuler_product}
\Ssq(fg)=\Ssq f\,\Ssq g+r_q(f,g)-(f-\Ssq f)(g-\Ssq g),
\end{equation}
where
\[
r_q(f,g)(x)=\lambda_{q+1}^3\intthree \widetilde h(\lambda_{q+1}y)
\big(f(x-y)-f(x)\big)\big(g(x-y)-g(x)\big)\,dy.
\]
The low-low term cancels exactly:
\[
\Ltwoinner{\Ssq\th\,\Ssq u,\Ssq\grad\th}
=\haf\intthree \Ssq u\cdot\grad |\Ssq\th|^2\,dx
=-\haf\intthree \div(\Ssq u)|\Ssq\th|^2\,dx=0.
\]
By \eqref{proof:Thm_non-negativity_NPEuler_translation}--\eqref{proof:Thm_non-negativity_NPEuler_cutoff},
\begin{align*}
\left|\Ltwoinner{r_q(\th,u),\Ssq\grad\th}\right|
&\les (K*d_u^2)^{\frac12}(q,s)(K*d_\th^2)(q,s),\\
\left|\Ltwoinner{(\th-\Ssq\th)(u-\Ssq u),\Ssq\grad\th}\right|
&\les (K*d_u^2)^{\frac12}(q,s)(K*d_\th^2)(q,s).
\end{align*}
Consequently,
\begin{align*}
\int_0^t \left|\Ltwoinner{\Ssq(\th u),\Ssq\grad\th}\right|ds
&\les \int_0^t (K*d_u^2)^{\frac12}(q,s)(K*d_\th^2)(q,s)\,ds \\
&\les \left(\int_0^t (K*d_u^3)(q,s)\,ds\right)^{\frac13}
\left(\int_0^t (K*d_\th^3)(q,s)\,ds\right)^{\frac23}.
\end{align*}
The first factor tends to zero by \eqref{Thm:Energy_equality_NPEuler_1} and $K\in\ell^1(\mathbb Z)$, while the second factor is uniformly bounded because $\th=\haf(\s-\rho)\in L^3([0,T];\Bb^{\frac13}_{3,\infty})$. Hence
\begin{equation}\label{proof:Thm_non-negativity_NPEuler_transport_limit}
\lim_{q\to\infty}\int_0^t \left|\Ltwoinner{\Ssq(\th u),\Ssq\grad\th}\right|ds=0.
\end{equation}

For the drift commutator, we again use \eqref{proof:Thm_non-negativity_NPEuler_product}:
\[
\Ssq(\th\grad\Phi)-\Ssq\th\,\Ssq\grad\Phi
=r_q(\th,\grad\Phi)-(\th-\Ssq\th)(\grad\Phi-\Ssq\grad\Phi).
\]
The elliptic relation $-\Delta\Phi=\rho$ and the Riesz transform bounds on the torus imply
\[
\|\nabla^2\Phi\|_{L^3}\les \threenorm{\rho}\les \|\rho\|_{\Bb^{\frac13}_{3,\infty}},
\]
and hence
\[
\threenorm{\grad\Phi(\cdot-y)-\grad\Phi(\cdot)}\les |y|\,\|\rho\|_{\Bb^{\frac13}_{3,\infty}},
\qquad
\threenorm{\grad\Phi-\Ssq\grad\Phi}\les \lambda_q^{-\frac43}\|\rho\|_{\Bb^{\frac13}_{3,\infty}}.
\]
Combining these bounds with \eqref{proof:Thm_non-negativity_NPEuler_translation} and \eqref{proof:Thm_non-negativity_NPEuler_cutoff}, we get
\begin{align*}
\left|\Ltwoinner{r_q(\th,\grad\Phi),\Ssq\grad\th}\right|
&\les \lambda_q^{-\frac23}(K*d_\th^2)(q,s)\|\rho(s)\|_{\Bb^{\frac13}_{3,\infty}},\\
\left|\Ltwoinner{(\th-\Ssq\th)(\grad\Phi-\Ssq\grad\Phi),\Ssq\grad\th}\right|
&\les \lambda_q^{-\frac23}(K*d_\th^2)(q,s)\|\rho(s)\|_{\Bb^{\frac13}_{3,\infty}}.
\end{align*}
Therefore, by \eqref{proof:Thm_non-negativity_NPEuler_Jensen} and H\"older's inequality in time,
\begin{align*}
&\int_0^t \left|\Ltwoinner{\Ssq(\th\grad\Phi)-\Ssq\th\,\Ssq\grad\Phi,\Ssq\grad\th}\right|ds \\
&\qquad\les \lambda_q^{-\frac23}
\left(\int_0^t (K*d_\th^3)(q,s)\,ds\right)^{\frac23}
\|\rho\|_{L^3([0,T];\Bb^{\frac13}_{3,\infty})}\to0.
\end{align*}
Thus
\begin{equation}\label{proof:Thm_non-negativity_NPEuler_drift_limit}
\lim_{q\to\infty}\int_0^t \left|\Ltwoinner{\Ssq(\th\grad\Phi)-\Ssq\th\,\Ssq\grad\Phi,\Ssq\grad\th}\right|ds=0.
\end{equation}

The last term in \eqref{proof:Thm_non-negativity_NPEuler_1} remains in the limiting energy identity. It is uniformly controlled by
\[
\int_0^t \left|\Ltwoinner{|\Ssq\th|^2,\Ssq\rho}\right|ds
\le \int_0^t \threenorm{\Ssq\rho}\threenorm{\Ssq\th}^2ds
\les \int_0^t \|\rho\|_{\Bb^{\frac13}_{3,\infty}}\|\th\|_{\Bb^{\frac13}_{3,\infty}}^2ds<\infty.
\]
It follows from \eqref{proof:Thm_non-negativity_NPEuler_1} and the preceding bounds that
\[
\sup_q\int_0^t\twonorm{\Ssq\grad\th(s)}^2ds<\infty,
\]
so $\th\in L^2([0,T];H^1(\T^3))$. Consequently, $\Ssq\grad\th\to\grad\th$ strongly in $L^2([0,T];L^2)$, and $\Ssq\rho\to\rho$, $\Ssq\th\to\th$ strongly in $L^3([0,T];L^3)$. Passing to the limit in \eqref{proof:Thm_non-negativity_NPEuler_1} yields
\begin{equation}\label{proof:Thm_non-negativity_NPEuler_theta_energy}
\haf \twonorm{\th(t)}^2 - \haf \twonorm{\th_0}^2
= -\int_0^t \twonorm{\grad\th}^2ds + \haf\int_0^t\Ltwoinner{\rho,|\th|^2}ds.
\end{equation}

Applying the same argument to \eqref{lemma:non-negativity_smooth_NPEuler_2} gives the opposite sign in the drift contribution:
\begin{equation}\label{proof:Thm_non-negativity_NPEuler_Theta_energy}
\haf \twonorm{\Theta(t)}^2 - \haf \twonorm{\Theta_0}^2
= -\int_0^t \twonorm{\grad\Theta}^2ds - \haf\int_0^t\Ltwoinner{\rho,|\Theta|^2}ds.
\end{equation}
In particular, $\Theta\in L^2([0,T];H^1(\T^3))$, and therefore $(u,\rho,\s)$ is parabolic.

We now prove non-negativity. Since $\th,\Theta\in L^2([0,T];H^1)$, the equations imply that the corresponding time derivatives lie in a suitable negative Sobolev space, for instance $L^1(0,T;H^{-1})$. Thus the standard convex-truncation chain rule allows us to test the actual equations with $-\th^-$ and $-\Theta^-$. This gives
\begin{align*}
\haf\twonorm{\th^-(t)}^2-\haf\twonorm{\th^-_0}^2
&=-\int_0^t\twonorm{\grad\th^-}^2ds+\haf\int_0^t\Ltwoinner{\rho,|\th^-|^2}ds,\\
\haf\twonorm{\Theta^-(t)}^2-\haf\twonorm{\Theta^-_0}^2
&=-\int_0^t\twonorm{\grad\Theta^-}^2ds-\haf\int_0^t\Ltwoinner{\rho,|\Theta^-|^2}ds.
\end{align*}
For $f=\th$ or $f=\Theta$, the drift term is bounded by
\[
\left|\intthree \rho |f^-|^2\,dx\right|
\le \twonorm{\rho}\fournorm{f^-}^2
\les \twonorm{\rho}\left(\twonorm{f^-}^{\frac12}\twonorm{\grad f^-}^{\frac32}+\twonorm{f^-}^2\right).
\]
Young's inequality gives
\[
\haf\twonorm{f^-(t)}^2
\le C\left(1+\|\rho\|_{L^\infty([0,T];L^2)}^4\right)
\int_0^t\twonorm{f^-(s)}^2ds.
\]
If $f^-_0=0$, Gr\"onwall's lemma implies $f^-=0$ on $[0,T]$. Hence $\th,\Theta\ge0$ whenever $\th_0,\Theta_0\ge0$.

Finally, multiplying the sum of \eqref{proof:Thm_non-negativity_NPEuler_theta_energy} and \eqref{proof:Thm_non-negativity_NPEuler_Theta_energy} by $2$ and using
\[
\Theta=\haf \left( \s+ \rho \right),
\qquad
\th=\haf \left( \s - \rho \right),
\qquad
\th^2-\Theta^2=-\rho\s,
\]
we obtain
\begin{align*}
\haf\left(\twonorm{\rho(t)}^2+\twonorm{\s(t)}^2\right)
-\haf\left(\twonorm{\rho_0}^2+\twonorm{\s_0}^2\right)
&= -\int_0^t\left(\twonorm{\grad\rho}^2+\twonorm{\grad\s}^2\right)ds \\
&\quad -\int_0^t\Ltwoinner{\s,|\rho|^2}ds,
\end{align*}
which is \eqref{Thm:non-negativity_NPEuler_1}. \qed

\section{The Nernst--Planck--Navier--Stokes equations}\label{sec:NPNL}

\noindent The electrodiffusion model coupled to a viscous incompressible fluid is governed by the Nernst--Planck--Navier--Stokes (NPNS) system:
\begin{subequations} \label{eq:NPNS}
\begin{align}
    \p_t \rho + u \cdot \nabla \rho &=   \Delta \rho + \nabla \s \cdot \nabla \Phi +\s \Delta \Phi, \label{eq:NPNS_1}  \\
    \p_t \s + u \cdot \nabla \s &=   \Delta \s + \nabla \rho \cdot \nabla \Phi +\rho \Delta \Phi,  \label{eq:NPNS_2} \\
    - \Delta \Phi &= \rho, \label{eq:NPNS_3}\\
    \p_t u + (u \cdot \nabla )u + \nabla p &=  \Delta u - \rho \nabla \Phi, \label{eq:NPNS_4} \\
    \nabla \cdot u &= 0.  \label{eq:NPNS_5}
\end{align}
\end{subequations}
Here we take the viscosity $\nu=1$ for simplicity.

\begin{Definition}\label{def:weak_solution_NPNS}
We call $(u,\rho, \s)$ a weak solution to the system \eqref{eq:NPNS} on $[0,T]$ with weakly divergence-free initial data $u_0 \in L^2 (\T^3)$ and $\rho_0, \s_0 \in L^2 (\T^3)$, satisfying the compatibility condition $\intthree \rho_0\,dx=0$, if 
\begin{align*}
u \in C_w ([0,T];L^2(\T^3)) \cap L^2([0,T];H^1(\T^3)), \ \ \rho,\s \in C_w ([0,T];L^2(\T^3)),
\end{align*}
and satisfies the following weak formulations. 
\begin{align}
&\Ltwoinner{u(t),\phi(t)} - \Ltwoinner{u_0,\phi(0)} - \int_0^t \Ltwoinner{u(s),\p_s \phi(s)} ds 
=\int_0^t \left( \Ltwoinner{ u \cdot \nabla \phi , u }  - \Ltwoinner{\grad u, \grad \phi}  -  \Ltwoinner{  \rho \nabla \Phi , \phi } \right) ds , \label {def:weak_solution_NPNS_1} \\
&\Ltwoinner{\rho(t),\psi_1(t)} - \Ltwoinner{\rho_0,\psi_1(0)} - \int_0^t \Ltwoinner{\rho(s),\p_s \psi_1(s) + u \cdot \nabla \psi_1 + \Delta \psi_1} ds  = - \int_0^t \Ltwoinner{\s \nabla \Phi, \nabla \psi_1}ds, \label {def:weak_solution_NPNS_2} \\
&\Ltwoinner{\s(t),\psi_2(t)} - \Ltwoinner{\s_0,\psi_2(0)} - \int_0^t \Ltwoinner{\s(s),\p_s \psi_2(s) + u \cdot \nabla \psi_2 + \Delta \psi_2} ds  = -  \int_0^t \Ltwoinner{\rho \nabla \Phi, \nabla \psi_2}ds, \label {def:weak_solution_NPNS_3}
\end{align}
\begin{equation}\label{def:weak_solution_NPNS_4}
\Ltwoinner{\rho+\Delta\Phi,\psi_3}=0.
\end{equation}
The potential is normalized by $\intthree \Phi(t,x)\,dx=0$. These identities hold for every $t\in[0,T]$ and every $\phi,\psi_1,\psi_2\in C^\infty([0,T]\times\T^3)$ such that $\phi$ is divergence-free and $\psi_3\in C^\infty(\T^3)$. In addition, $\nabla\cdot u=0$ in $\mathcal D'((0,T)\times\T^3)$.\\

We further call a weak solution a \textit{Leray--Hopf solution} if it satisfies the following energy inequality: 
\begin{align}
& \twonorm{u(t)}^2   +  \twonorm{\nabla \Phi(t)}^2  + 2 \int_0^t \twonorm{\grad u}^2 ds +2  \int_0^t \twonorm{\Delta \Phi}^2 ds \notag\\
+\ & 2 \int_0^t \int_{\T^3} \s|\nabla \Phi|^2 dx ds
 \leq \twonorm{u_0}^2 + \twonorm{\nabla \Phi_0}^2, \ \ \ \forall t \in (0,T], \label{eq:NPNS_Energy_inequality}
\end{align}
where $\Phi_0 = (-\Delta)^{-1} \rho_0$ such that $\int_{\mathbb T^3} \Phi_0 dx =0$.
\end{Definition}
\begin{Remark}\label{rem:sigma_mass_conservation_NPNS}
Every weak solution in Definition \ref{def:weak_solution_NPNS} conserves the spatial average of the total concentration $\s$. As above, this is justified by regularizing the distributional equation before integrating in space. From \eqref{def:weak_solution_NPNS_3},
\[
\p_t\s+\divv(u\s)=\Delta\s+\divv(\rho\nabla\Phi)
\quad\text{in }\mathcal D'((0,T)\times\T^3).
\]
Let $\mathcal J_\epsilon$ be a periodic Friedrichs mollifier in $x$. Then
\[
\p_t \mathcal J_\epsilon\s+\divv\mathcal J_\epsilon(u\s)
=\Delta \mathcal J_\epsilon\s+\divv\mathcal J_\epsilon(\rho\nabla\Phi)
\quad\text{in }\mathcal D'((0,T)\times\T^3).
\]
Integrating this smooth-in-space identity over $\T^3$ gives
\[
\frac{d}{dt}\intthree \mathcal J_\epsilon\s(t,x)\,dx=0
\quad\text{in }\mathcal D'(0,T).
\]
Since $\mathcal J_\epsilon$ preserves the zero Fourier mode, this implies
\[
\frac{d}{dt}\intthree \s(t,x)\,dx=0
\quad\text{in }\mathcal D'(0,T).
\]
Using $\s\in C_w([0,T];L^2(\T^3))$ to identify the continuous representative in time, we obtain
\[
\intthree \s(t,x)\,dx=\intthree \s_0(x)\,dx=:C,
\qquad t\in[0,T].
\]
Thus, in the sequel, we may assume $\intthree \s(t,x)\,dx=C$ for weak solutions of \eqref{eq:NPNS}.
\end{Remark}

The following uniqueness criterion is the main result of this section.
\begin{Theorem}\label{Thm:Serrin_NPNL}
Let $(u,\rho^u,\s^u)$ and $(v,\rho^v,\s^v)$ be two Leray--Hopf solutions on $[0,T]$ of the system \eqref{eq:NPNS} with the same initial data.
If the velocity field $u$ satisfies the Serrin condition
\begin{equation}\label{Thm:Serrin_NPNL_1}
u \in L^p([0,T];L^q(\T^3)),
\qquad
\frac{2}{p} + \frac{3}{q} = 1, \qquad (p,q) \in [2,\infty) \times (3,\infty],
\end{equation}
then $u=v$, $\rho^u=\rho^v$, and $\s^u=\s^v$ on $[0,T]$.
\end{Theorem}

\begin{Remark}
The condition on the exponents \(p\) and \(q\) in \eqref{Thm:Serrin_NPNL_1} is known as the Prodi--Serrin condition for the Navier--Stokes equations. Related sharpness and nonuniqueness phenomena for Navier--Stokes are discussed in \cite{CL22}.
\end{Remark}

\begin{Lemma}[Cross identities under the Serrin condition]\label{lem:Serrin_cross_identities}
Let $(u,\rho^u,\s^u)$ and $(v,\rho^v,\s^v)$ be two Leray--Hopf solutions of \eqref{eq:NPNS} on $[0,T]$. Assume that $u$ satisfies the Serrin condition \eqref{Thm:Serrin_NPNL_1}. Then, for every $t\in[0,T]$, the maps
\[
 s\mapsto \Ltwoinner{u(s),v(s)},
 \qquad
 s\mapsto \Ltwoinner{\grad\Phi^u(s),\grad\Phi^v(s)}
\]
have absolutely continuous representatives, and the following identities hold:
\begin{align}
\Ltwoinner{u(t),v(t)}
&=\Ltwoinner{u_0,v_0}-2\int_0^t\intthree \grad u:\grad v\,dxds
-\int_0^t\Ltwoinner{((u-v)\cdot\grad)(u-v),u}\,ds \notag\\
&\quad -\int_0^t\intthree \left(\rho^u v\cdot\grad\Phi^u+\rho^v u\cdot\grad\Phi^v\right)\,dxds, \label{lemma:Serrin_cross_velocity}
\end{align}
and
\begin{align}
\Ltwoinner{\grad\Phi^u(t),\grad\Phi^v(t)}
&=\Ltwoinner{\grad\Phi^u_0,\grad\Phi^v_0}
-2\int_0^t\Ltwoinner{\rho^u,\rho^v}\,ds \notag\\
&\quad +\int_0^t\intthree \left(\rho^u u\cdot\grad\Phi^v+\rho^v v\cdot\grad\Phi^u\right)\,dxds \notag\\
&\quad -\int_0^t\intthree (\s^u+\s^v)\grad\Phi^u\cdot\grad\Phi^v\,dxds. \label{lemma:Serrin_cross_potential}
\end{align}
\end{Lemma}

\begin{proof}
We justify the identities by a standard approximation argument. Let $P_N$ be a Fourier projection and let $\eta_\varepsilon$ be a time mollifier. In the weak formulation for $u$ we use the divergence-free test function $\eta_\varepsilon*(P_Nv)$, and in the weak formulation for $v$ we use $\eta_\varepsilon*(P_Nu)$; a smooth time cut-off approximating $\mathbbm{1}_{[0,t]}$ is inserted and then removed. The linear terms and the force terms pass to the limit by the Leray--Hopf bounds. The only terms that require the Serrin condition are the convective terms. For $w,z\in\{u,v\}$, the interpolation inequality
\[
\|w\|_{L^{\frac{2q}{q-2}}}
\le C\twonorm{w}^{1-\frac{3}{q}}\twonorm{\grad w}^{\frac{3}{q}}
\]
(with the usual interpretation when $q=\infty$) and \eqref{Thm:Serrin_NPNL_1} give
\begin{align*}
\int_0^T \qnorm{u}\,\|w\|_{L^{\frac{2q}{q-2}}}\twonorm{\grad z}\,ds
&\le C\|u\|_{L^p_tL^q_x}\|w\|_{L^\infty_tL^2_x}^{1-\frac{3}{q}}
\|\grad w\|_{L^2_tL^2_x}^{\frac{3}{q}}
\|\grad z\|_{L^2_tL^2_x}<\infty.
\end{align*}
Thus the products $(u\cdot\grad) v\cdot u$ and $(v\cdot\grad )u\cdot v$ are integrable, and the limit of the regularized weak formulations gives, in the sense of distributions in time,
\[
\frac{d}{dt}\Ltwoinner{u,v}
=-2\intthree\grad u:\grad v\,dx
+\intthree (u\cdot\grad v)\cdot u\,dx
+\intthree (v\cdot\grad u)\cdot v\,dx
-\intthree \left(\rho^u v\cdot\grad\Phi^u+\rho^v u\cdot\grad\Phi^v\right)dx.
\]
Since $\divv u=\divv v=0$, by integration by parts the two convective terms reduce to
\[
\intthree (u\cdot\grad v)\cdot u\,dx+
\intthree (v\cdot\grad u)\cdot v\,dx
=-\Ltwoinner{((u-v)\cdot\grad)(u-v),u}.
\]
The right-hand side belongs to $L^1(0,T)$, so $\Ltwoinner{u,v}$ is absolutely continuous and the fundamental theorem of calculus gives \eqref{lemma:Serrin_cross_velocity}.

For the second identity, we use the same regularization in the two weak formulations for $\rho^u$ and $\rho^v$, with test functions $\eta_\varepsilon*(P_N\Phi^v)$ and $\eta_\varepsilon*(P_N\Phi^u)$, respectively. Passing to the limit is justified by
\[
\rho^u,\rho^v,\s^u,\s^v\in L^\infty(0,T;L^2),\quad
\grad\Phi^u,\grad\Phi^v\in L^\infty(0,T;L^6),\quad
u,v \in L^2(0,T;H^1)\subseteq L^1(0,T;L^3)
\]
for the Leray--Hopf velocities $u,v$, together with the elliptic estimate $\|\grad\Phi\|_{L^6}\le C\twonorm{\rho}$. Hence, distributionally in time,
\begin{align*}
\frac{d}{dt}\Ltwoinner{\grad\Phi^u,\grad\Phi^v}
&=-2\Ltwoinner{\rho^u,\rho^v}
+\intthree \left(\rho^u u\cdot\grad\Phi^v+\rho^v v\cdot\grad\Phi^u\right)dx \\
&\quad -\intthree (\s^u+\s^v)\grad\Phi^u\cdot\grad\Phi^v\,dx.
\end{align*}
The right-hand side belongs to $L^1(0,T)$, and the fundamental theorem of calculus yields \eqref{lemma:Serrin_cross_potential}.
\end{proof}

\begin{proof}[Proof of Theorem \ref{Thm:Serrin_NPNL}]
Let $(u,\rho^u,\s^u)$ and $(v,\rho^v,\s^v)$ be two Leray--Hopf solutions on $[0,T]$ with the same initial conditions $(u_0,\rho^u_0,\s^u_0)=(v_0,\rho^v_0,\s^v_0)$. We apply the relative energy method (see \cite{BV22}). Denote
\[
\om := u-v,\qquad \barphi := \Phi^u-\Phi^v,\qquad \barrho := \rho^u-\rho^v,\qquad \barsgm := \s^u-\s^v,
\]
where
\[
\rho^u = -\Delta \Phi^u, \ \ \rho^v = -\Delta \Phi^v.
\]
The equal initial data and conservation of the $\rho$- and $\s$-masses give
\[
\intthree \barrho(t,x)\,dx=0,\qquad \intthree \barsgm(t,x)\,dx=0,
\]
for a.e. $t\in[0,T]$. With the normalization $\intthree\barphi(t,x)\,dx=0$, the operators $\lb^{-2}=(-\Delta)^{-1}$ and $\lb^{-1}$ are therefore well-defined on these difference variables.
Using the energy inequalities \eqref{eq:NPNS_Energy_inequality} for both solutions yields
\begin{align}
&\haf \twonorm{\om(t)}^2 + \int_0^t \twonorm{\grad \om}^2 ds + \haf \twonorm{\grad \barphi(t)}^2 + \int_0^t \twonorm{\Delta \barphi}^2 ds \notag \\
+\  &\int_0^t \intthree \s^u(x,s) |\grad \barphi|^2 dxds + \int_0^t \intthree \s^v(x,s) |\grad \barphi|^2 dxds \notag \\
\leq &\  \twonorm{u_0}^2 + \twonorm{\grad \Phi^u_0}^2 - \Ltwoinner{u(t),v(t)}
-2 \int_0^t \intthree \grad u : \grad v dx ds - \Ltwoinner{\grad \Phi^u(t), \grad \Phi^v(t)} - 2 \int_0^t \Ltwoinner{\rho^u, \rho^v} ds \notag \\
+\  &\int_0^t \intthree \left(  \s^u(x,s) |\grad \Phi^v|^2 + \s^v(x,s) |\grad \Phi^u|^2 \right) dxds -2 \int_0^t \intthree \left( \s^u + \s^v \right) \grad \Phi^u \cdot \grad \Phi^v dx ds \label{proof_Thm_Serrin_NPNL_1}
\end{align}
By Lemma \ref{lem:Serrin_cross_identities}, applied on the interval $[0,t]$, the cross terms in \eqref{proof_Thm_Serrin_NPNL_1} satisfy
\begin{align}
\Ltwoinner{u(t),v(t)}
&=\twonorm{u_0}^2-2\int_0^t\intthree \grad u:\grad v\,dxds
-\int_0^t\Ltwoinner{(\om\cdot\grad)\om,u}\,ds \notag\\
&\quad -\int_0^t\intthree \left( \rho^u v\cdot\grad \Phi^u + \rho^v u\cdot\grad \Phi^v \right)\,dxds, \label{proof_Thm_Serrin_NPNL_cross_u}
\end{align}
and
\begin{align}
\Ltwoinner{\grad \Phi^u(t),\grad \Phi^v(t)}
&=\twonorm{\grad \Phi^u_0}^2-2\int_0^t\Ltwoinner{\rho^u,\rho^v}\,ds \notag\\
&\quad +\int_0^t\intthree \left(\rho^u u\cdot\grad \Phi^v+\rho^v v\cdot\grad \Phi^u\right)\,dxds \notag\\
&\quad -\int_0^t\intthree (\s^u+\s^v)\grad \Phi^u\cdot\grad \Phi^v\,dxds. \label{proof_Thm_Serrin_NPNL_cross_phi}
\end{align}
The inequality \eqref{proof_Thm_Serrin_NPNL_1} then becomes
\begin{align}
    &\haf \twonorm{\om(t)}^2 + \int_0^t \twonorm{\grad \om}^2 ds + \haf \twonorm{\grad \barphi(t)}^2 + \int_0^t \twonorm{\Delta \barphi}^2 ds \notag \\
+\ &\int_0^t \intthree \s^u(x,s) |\grad \barphi|^2 dxds + \int_0^t \intthree \s^v(x,s) |\grad \barphi|^2 dxds \notag \\
\leq \ &\int_0^t \Ltwoinner{(\om \cdot \grad)\om,u} ds + \int_0^t \intthree \left( \rho^u v \cdot \grad \Phi^u + \rho^v u \cdot \grad \Phi^v - \rho^u u \cdot \grad \Phi^v - \rho^v v \cdot \grad \Phi^u \right) dx ds \notag \\
+\ &\int_0^t \intthree \left( \s^u |\grad \Phi^v|^2 + \s^v |\grad \Phi^u|^2 \right) dxds - \int_0^t \intthree (\s^u + \s^v) \grad \Phi^u \cdot \grad \Phi^v dx ds \notag \\
=\  &\int_0^t \Ltwoinner{(\om \cdot \grad)\om,u} ds + \int_0^t \intthree \left(\om \cdot \grad \Phi^u - u \cdot \grad \barphi \right) \Delta \barphi dx ds \ + \int_0^t \intthree \s^v |\grad \barphi|^2 - \barsgm \grad \Phi^v \cdot \grad \barphi dx ds \notag \\
:=\  &I_1 + I_2 + I_3.  \label{proof_Thm_Serrin_NPNL_2}
\end{align}
We now estimate \(I_1, I_2,\) and \(I_3\). For \(I_1\), we apply the Gagliardo--Nirenberg inequality to obtain
\begin{align}
|I_1| &\leq \int_0^t \|u \|_{L^q(\T^3)} \| \om \|_{L^{\frac{2q}{q-2}}(\T^3)} \twonorm{\grad \om} ds \les \int_0^t \|u \|_{L^q(\T^3)} \twonorm{\om}^{1-\frac{3}{q}} \twonorm{\grad \om}^{1+\frac{3}{q}} ds \notag \\
&\leq C \int_0^t \|u \|^p_{L^q(\T^3)} \twonorm{\om}^2 ds + \frac{1}{4} \int_0^t \twonorm{\grad \om}^2 ds \label{proof_Thm_Serrin_NPNL_3}.
\end{align}
For $I_2$, we similarly have
\begin{align}
|I_2| &\leq \int_0^t \intthree \left|\left(\om \cdot \grad \Phi^u - u \cdot \grad \barphi \right) \Delta \barphi\right| dx ds  \notag  \\
&\leq C \int_0^t \twonorm{\rho^u}^4 \twonorm{\om}^2 ds + \frac{1}{8} \int_0^t \twonorm{\Delta \barphi}^2 ds + \frac{1}{4} \int_0^t \twonorm{\grad \om}^2  ds \notag \\
&+ C \int_0^t \|u \|^p_{L^q(\T^3)} \twonorm{\grad \barphi}^2 ds + \frac{1}{8} \int_0^t \twonorm{\Delta \barphi}^2 ds. \label{proof_Thm_Serrin_NPNL_4}
\end{align}
Likewise,
\begin{align}
|I_3| &\leq \int_0^t \twonorm{\s^v}\fournorm{\grad\barphi}^2ds
+\int_0^t\twonorm{\barsgm}\sixNorm{\grad\Phi^v}\threenorm{\grad\barphi}ds \notag\\
&\leq C\int_0^t\left(\twonorm{\rho^v}^4+\twonorm{\s^v}^4\right)\twonorm{\grad\barphi}^2ds
+\frac14\int_0^t\twonorm{\Delta\barphi}^2ds
+\frac14\int_0^t\twonorm{\barsgm}^2ds. \label{proof_Thm_Serrin_NPNL_5}
\end{align}
Here we used $\fournorm{\grad\barphi}^2\lesssim \twonorm{\grad\barphi}^{1/2}\twonorm{\Delta\barphi}^{3/2}$, $\threenorm{\grad\barphi}\lesssim\twonorm{\grad\barphi}^{1/2}\twonorm{\Delta\barphi}^{1/2}$, the Sobolev embedding $H^1(\T^3)\hookrightarrow L^6(\T^3)$, and the elliptic estimate $\sixNorm{\grad\Phi^v}\lesssim\twonorm{\rho^v}$. To control the last term in \eqref{proof_Thm_Serrin_NPNL_5}, we consider the weak formulation for $\s^u$ and $\s^v$. By subtracting the equation \eqref{def:weak_solution_NPNS_3} of $\s^v$ from that of $\s^u$ and testing the zero-mean equation with $\lb^{-2}\barsgm=(-\Delta)^{-1}\barsgm$, we obtain 
\begin{align}
\haf \twonorm{\lb^{-1} \barsgm(t)}^2 + \int_0^t \twonorm{\barsgm}^2 ds &= \int_0^t \Ltwoinner{\s^v \om + \barsgm u, \grad \lb^{-2} \barsgm} ds - \int_0^t \Ltwoinner{\barrho \grad \Phi^v + \rho^u \grad \barphi, \grad \lb^{-2} \barsgm}ds \notag \\
&:=J_1 + J_2, \notag
\end{align}
where the Fourier multiplier $\lb^s$ is defined on zero-mean functions by
\[
\widehat {\lb^s f} (k) = |k|^s \hat {f} (k),\qquad k\ne0.
\]
We estimate $J_1$ and $J_2$ in the same fashion as $I_1$, $I_2$, and $I_3$:
\begin{align}
    |J_1| &\leq \int_0^t \twonorm{\s^v} \threenorm{\om} \sixNorm{\grad \lb^{-2} \barsgm} ds + \int_0^t \qnorm{u} \twonorm{\barsgm} \| \grad \lb^{-2} \barsgm \|_{L^{\frac{2q}{q-2}}(\T^3)} ds  \notag \\
    &\les \int_0^t \twonorm{\s^v} \twonorm{\om}^\haf \twonorm{\grad \om}^\haf \twonorm{\barsgm} ds + \int_0^t \qnorm{u} \twonorm{\lb^{-1} \barsgm}^{1-\frac{3}{q}} \twonorm{\barsgm}^{1+\frac{3}{q}} ds   \notag \\
    &\leq C \int_0^t \twonorm{\s^v}^2 \twonorm{\om} \twonorm{\grad \om} ds + \frac{1}{4} \int_0^t \twonorm{\barsgm}^2 ds \notag \\
    &+ C \int_0^t \qnorm{u}^p \twonorm{\lb^{-1} \barsgm}^2 ds +\frac{1}{4} \int_0^t \twonorm{\barsgm}^2 ds \notag \\
    &\leq C \int_0^t \twonorm{\s^v}^4 \twonorm{\om}^2 ds + \frac{1}{4} \int_0^t \twonorm{\grad \om}^2 ds  \notag \\
    & + C \int_0^t \qnorm{u}^p \twonorm{\lb^{-1} \barsgm}^2 ds + \haf \int_0^t \twonorm{\barsgm}^2 ds, \notag
\end{align}
\begin{align}
|J_2| &\leq \int_0^t \twonorm{\barrho} \threenorm{\grad \lb^{-2} \barsgm} \sixNorm{\grad \Phi^v} ds + \int_0^t \twonorm{\rho^u} \sixNorm{\grad \barphi} \threenorm{\grad \lb^{-2} \barsgm} ds  \notag \\
&\leq \int_0^t \left( \twonorm{\rho^u} + \twonorm{\rho^v} \right) \twonorm{\lb^{-1} \barsgm}^\haf \twonorm{\barsgm}^\haf \twonorm{\Delta \barphi} ds \notag \\
&\leq C \int_0^t \left( \twonorm{\rho^u}^4 + \twonorm{\rho^v}^4 \right) \twonorm{\lb^{-1} \barsgm}^2 ds + \frac{1}{4} \int_0^t \twonorm{\barsgm}^2 ds + \frac{1}{8} \int_0^t \twonorm{\Delta \barphi}^2 ds  \notag
\end{align}
Consequently,
\begin{align}
\haf \twonorm{\lb^{-1} \barsgm(t)}^2 + \int_0^t \twonorm{\barsgm}^2 ds &\leq  C \int_0^t \twonorm{\s^v}^4 \twonorm{\om}^2 ds + \frac{1}{4} \int_0^t \twonorm{\grad \om}^2 ds  \notag \\
    & + C \int_0^t \qnorm{u}^p \twonorm{\lb^{-1} \barsgm}^2 ds + \haf \int_0^t \twonorm{\barsgm}^2 ds \notag \\
    & + C \int_0^t \left( \twonorm{\rho^u}^4 + \twonorm{\rho^v}^4 \right) \twonorm{\lb^{-1} \barsgm}^2 ds \notag \\
    & + \frac{1}{4} \int_0^t \twonorm{\barsgm}^2 ds + \frac{1}{8} \int_0^t \twonorm{\Delta \barphi}^2 ds. \label{proof_Thm_Serrin_NPNL_6}
\end{align}
Finally, notice that 
\begin{align}
    \int_0^t \intthree |\s^u(x,s)| |\grad \barphi|^2 dxds &\leq \int_0^t \twonorm{\s^u} \fournorm{\grad \barphi}^2 ds \notag \\
    &\leq  \int_0^t \twonorm{\s^u} \twonorm{\grad \barphi}^\haf \twonorm{\Delta \barphi}^{\frac{3}{2}} ds \notag \\
    &\leq C \int_0^t \twonorm{\s^u}^4 \twonorm{\grad \barphi}^2 ds + \frac{1}{4} \int_0^t \twonorm{\Delta \barphi}^2 ds \label{proof_Thm_Serrin_NPNL_7}
\end{align}
and similarly,
\begin{equation}\label{proof_Thm_Serrin_NPNL_8}
\int_0^t \intthree |\s^v(x,s)| |\grad \barphi|^2 dxds \leq C \int_0^t \twonorm{\s^v}^4 \twonorm{\grad \barphi}^2 ds + \frac{1}{4} \int_0^t \twonorm{\Delta \barphi}^2 ds
\end{equation}
Set
\[
\xi(t) := \twonorm{\om(t)}^2 + \twonorm{\grad \barphi}^2 + \twonorm{\lb^{-1} \barsgm}^2.
\]
Combining \eqref{proof_Thm_Serrin_NPNL_2}--\eqref{proof_Thm_Serrin_NPNL_8}, we conclude that 
\begin{align*}
&\haf \xi(t) + \int_0^t \twonorm{\grad \om}^2 ds + \int_0^t \twonorm{\Delta \barphi}^2 ds + \int_0^t \twonorm{\barsgm}^2 ds \\
\leq\ &C \int_0^t \qnorm{u}^p \xi(s) ds + C \int_0^t \left( \twonorm{\rho^u}^4 + \twonorm{\rho^v}^4 + \twonorm{\s^u}^4 + \twonorm{\s^v}^4 \right) \xi(s) ds \\
-\ &\int_0^t \intthree \s^u(x,s) |\grad \barphi|^2 dxds - \int_0^t \intthree \s^v(x,s) |\grad \barphi|^2 dxds\\ 
+\ &\frac{3}{4} \int_0^t \twonorm{\grad \om}^2 ds + \haf \int_0^t \twonorm{\Delta \barphi}^2 ds + \int_0^t \twonorm{\barsgm}^2 ds \\
\leq\ &C \int_0^t \qnorm{u}^p \xi(s) ds + C \int_0^t \left( \twonorm{\rho^u}^4 + \twonorm{\rho^v}^4 + \twonorm{\s^u}^4 + \twonorm{\s^v}^4 \right) \xi(s) ds \\
+\ &\frac{3}{4} \int_0^t \twonorm{\grad \om}^2 ds + \int_0^t \twonorm{\Delta \barphi}^2 ds + \int_0^t \twonorm{\barsgm}^2 ds.
\end{align*}

Subtracting the dissipation terms on the right-hand side from the left-hand side gives
\[
\xi(t)\le C\int_0^t a(s)\xi(s)\,ds,
\]
where
\[
a(s)=\qnorm{u(s)}^p+\twonorm{\rho^u(s)}^4+\twonorm{\rho^v(s)}^4+\twonorm{\s^u(s)}^4+\twonorm{\s^v(s)}^4\in L^1(0,T).
\]
Gr\"onwall's inequality gives $\xi(t)\equiv0$ for all $t\in[0,T]$, which completes the proof.
\end{proof}

\begin{Definition}\label{def:Leray_Hopf&Parabolic_NPNS}
A weak solution $(u,\rho,\s)$ to the system \eqref{eq:NPNS} on $[0,T]$ is called \textit{parabolic} if
\[
\rho,\s \in L^2([0,T];H^1(\T^3)).
\]
\end{Definition}
Next, we establish a regularity criterion implying parabolicity in Theorem~\ref{Thm:Regularity_NPNS}, and then further prove regularity criteria that ensure non-negativity of the ion concentrations in Corollary~\ref{Coro:non-negativity_NPNS}.

\begin{Theorem}\label{Thm:Regularity_NPNS}
Let $(u,\rho,\s)$ be a weak solution of \eqref{eq:NPNS}, and let $p,q \in [2,\infty]$.
Assume in addition that
\[
u, \grad \Phi \in L^p(0,T;L^q(\T^3)).
\]
If there exist $r,s \in [2,\infty]$ such that
\[
\rho,\s \in L^r(0,T;L^s(\T^3)),
\]
where the exponents satisfy
\begin{equation}\label{Thm:Regularity_NPNS_1}
\frac{1}{q} + \frac{1}{s} = \haf, \qquad \frac{1}{p} + \frac{1}{r} = \haf.
\end{equation}
Then the solution is parabolic.
\end{Theorem}

To prove the theorem above, we invoke the following standard product commutator estimate.
\begin{Lemma}\label{lem:product_commutator}
Let $1\le p,q,r,s\le\infty$ satisfy
\[
\frac1p+\frac1r=\frac12,
\qquad
\frac1q+\frac1s=\frac12.
\]
If $f\in L^r(0,T;L^s(\T^3))$ and $g\in L^p(0,T;L^q(\T^3))$, then
\[
\Ss_j(fg)-\Ss_j f\,\Ss_j g\longrightarrow0
\quad\text{strongly in }L^2(0,T;L^2(\T^3)).
\]
\end{Lemma}

\begin{proof}
Let
\[
K_j(y):=\l_{j+1}^3\widetilde h(\l_{j+1}y),\qquad
\delta_y f(x,t):=f(x-y,t)-f(x,t).
\]
The product identity
\[
\Ss_j(fg)=\Ss_j f\,\Ss_j g+r_j(f,g)-(f-\Ss_jf)(g-\Ss_jg),
\]
where
\[
r_j(f,g)(x,t)=\intthree K_j(y)\,\delta_y f(x,t)\,\delta_y g(x,t)\,dy,
\]
shows that it is enough to prove
\[
r_j(f,g)\to0,
\qquad
(f-\Ss_jf)(g-\Ss_jg)\to0
\]
strongly in $L^2(0,T;L^2(\T^3))$.

We first assume that $r,s<\infty$. By H\"older's inequality,
\begin{align*}
\|r_j(f,g)\|_{L^2_tL^2_x}
&\le \intthree |K_j(y)|\,
\|\delta_y f\|_{L^r_tL^s_x}\,
\|\delta_y g\|_{L^p_tL^q_x}\,dy  \\
&\les \int_{\mathbb R^3}|\widetilde h(z)|\,
\|\delta_{\l_{j+1}^{-1}z} f\|_{L^r_tL^s_x}\,
\|\delta_{\l_{j+1}^{-1}z} g\|_{L^p_tL^q_x}\,dz .
\end{align*}
For each fixed $z$, the translation continuity in $L^r(0,T;L^s(\T^3))$ and in $L^p(0,T;L^q(\T^3))$ gives
\[
\|\delta_{\l_{j+1}^{-1}z} f\|_{L^r_tL^s_x}\to0,
\qquad
\|\delta_{\l_{j+1}^{-1}z} g\|_{L^p_tL^q_x}\to0,
\]
as $j \to \infty$. Moreover the integrand is bounded by
\[
4|\widetilde h(z)|\,
\|f\|_{L^r_tL^s_x}\|g\|_{L^p_tL^q_x},
\]
which is integrable in $z$. Hence $r_j(f,g)\to0$ in $L^2_tL^2_x$ by dominated convergence. Similarly,
\begin{align*}
\|(f-\Ss_jf)(g-\Ss_jg)\|_{L^2_tL^2_x}
&\le \|f-\Ss_jf\|_{L^r_tL^s_x}
   \|g-\Ss_jg\|_{L^p_tL^q_x}
\to0,
\end{align*}
because $\Ss_j\to I$ strongly on mixed Lebesgue spaces with finite exponents.

It remains to consider the cases in which one exponent is infinite. Unlike the finite exponents case, the difference in the $L^\infty$ norm need not converge to zero. Instead, it is enough to show that it is uniformly bounded while the paired finite $L^2$-type difference converges to zero. Assume first that $q=\infty$ and $s=2$. For fixed $z$, using the boundedness of translations in $L^\infty_x$ we have
\begin{align*}
\|\delta_{\l_{j+1}^{-1}z} f\|_{L^2_x}^2
  \|\delta_{\l_{j+1}^{-1}z} g\|_{L^\infty_x}^2  \le 16\|f\|_{L^2_x}^2\|g\|_{L^\infty_x}^2.
\end{align*}
The right-hand side is integrable in time since
$f\in L^r_tL^2_x$, $g\in L^p_tL^\infty_x$, and $1/r+1/p=1/2$.
For a.e. $t$, translations are continuous in $L^2_x$, so
\[
\|\delta_{\l_{j+1}^{-1}z} f(t)\|_{L^2_x}\to0.
\]
Therefore, by dominated convergence in time,
\[
\left(\int_0^T
\|\delta_{\l_{j+1}^{-1}z} f\|_{L^2_x}^2
\|\delta_{\l_{j+1}^{-1}z} g\|_{L^\infty_x}^2\,dt\right)^{1/2} \to	0.
\]
This quantity is bounded by
$4\|f\|_{L^r_tL^2_x}\|g\|_{L^p_tL^\infty_x}$, so dominated convergence in $z$ gives
$r_j(f,g)\to0$ in $L^2_tL^2_x$. For the second term,
\begin{align*}
\|(f-\Ss_jf)(g-\Ss_jg)\|_{L^2_tL^2_x}^2
&\le C\int_0^T
\|f-\Ss_jf\|_{L^2_x}^2\|g\|_{L^\infty_x}^2\,dt
\to0,
\end{align*}
again by dominated convergence, since $\Ss_j f(t)\to f(t)$ in $L^2_x$ for a.e. $t$ and
\[
\|f-\Ss_jf\|_{L^2_x}^2\|g\|_{L^\infty_x}^2
\le C\|f\|_{L^2_x}^2\|g\|_{L^\infty_x}^2\in L^1(0,T).
\]
The case $s = \infty$ follow by exchanging $f$ and $g$. This completes the proof.
\end{proof}

\begin{proof}[Proof of Theorem \ref{Thm:Regularity_NPNS}]
Using the equation \eqref{lemma:non-negativity_smooth_NPEuler_1}, which remains valid for the NPNS system, and testing it with $\Ss_j^2 \th$, we obtain
\begin{align}
    &\haf \twonorm{\Ss_j \th (t)}^2 - \haf \twonorm{\Ss_j \th_0}^2 +  \int_0^t \twonorm{\Ss_j \nabla \th(s)}^2 ds = \int_0^t \Ltwoinner{\Ss_j \left( (\th u) + \th \grad \Phi \right), \Ss_j \grad \th} ds \notag \\
    =\ &\int_0^t \Ltwoinner{\Ssj (\th u), \Ssj \grad \th} ds + \int_0^t \Ltwoinner{\Ssj (\th \grad \Phi) - \Ssj \th \Ssj \grad \Phi, \Ssj \grad \th} ds  + \haf \int_0^t \Ltwoinner{|\Ssj \th|^2, \Ssj \rho} ds. \label{proof:Thm_Regularity_NPNS_1}
\end{align}
Using the Ladyzhenskaya inequality, we bound the last term of the right-hand side of \eqref{proof:Thm_Regularity_NPNS_1} by
\begin{align*}
 \int_0^t \Ltwoinner{|\Ssj \th|^2, \Ssj \rho} ds &\leq  \int_0^t \twonorm{\Ssj \rho} \fournorm{\Ssj \th}^2 ds \\
&\leq  C\int_0^t \twonorm{\Ssj \rho}\left(\twonorm{\Ssj \th}^\haf \twonorm{\Ssj \grad \th}^{\frac{3}{2}}+\twonorm{\Ssj\th}^2\right) ds \\
&\leq C \int_0^t \left(1+\twonorm{\Ssj \rho}^4\right) \twonorm{\Ssj \th}^2 ds + \frac14 \int_0^t \twonorm{\Ssj \grad \th}^2 ds.
\end{align*}
By the Cauchy-Schwarz and Young inequalities, 
\eqref{proof:Thm_Regularity_NPNS_1} becomes
\begin{align}
    &\twonorm{\Ss_j \th (t)}^2 - \twonorm{\Ss_j \th_0}^2 +  \int_0^t \twonorm{\Ss_j \nabla \th(s)}^2 ds \notag \\
    \leq \ &C\int_0^t \twonorm{\Ssj (\th u)}^2 ds + C \int_0^t \twonorm{\Ssj (\th \grad \Phi) - \Ssj \th \Ssj \grad \Phi}^2 ds \notag \\
    +\ &C \int_0^t \left(1+\twonorm{\Ssj \rho}^4\right) \twonorm{\Ssj \th}^2 ds
    =: K_1 + K_2 + K_3. \label{proof:Thm_Regularity_NPNS_2}
\end{align}
One can bound $K_1$ simply by
\begin{align}
 K_1\les \int_0^t \|\theta u\|_{L^2}^2 ds \les \|\theta\|_{L^r_t L^s_x}^2\|u\|_{L^p_tL_x^q}^2 < \infty.\label{proof:Thm_Regularity_NPNS_3}
\end{align} The same argument remains valid if either $q$ or $s$ is equal to infinity. For $K_2$, Lemma \ref{lem:product_commutator} applied to $(\th,\nabla\Phi)$, together with
\[
\th\in L^r(0,T;L^s),\qquad \nabla\Phi\in L^p(0,T;L^q),\qquad
\frac1r+\frac1p=\frac12,\quad \frac1s+\frac1q=\frac12,
\]
gives
\begin{equation}\label{proof:Thm_Regularity_NPNS_4}
\limsup_{j \to \infty} K_2 = 0.
\end{equation}
For $K_3$, we have
\begin{equation}\label{proof:Thm_Regularity_NPNS_5}
\limsup_{j \to \infty} K_3 \le C\int_0^t \left(1+\twonorm{\rho}^4\right) \twonorm{\th}^2 ds < \infty .
\end{equation}
Finally, rearranging \eqref{proof:Thm_Regularity_NPNS_2} and sending $j \to \infty$, we conclude that
\begin{align*}
   \limsup_{j \to \infty}  \int_0^t \twonorm{\Ssj \nabla \th(s)}^2 ds < \infty.
\end{align*}
This implies that $\th\in L^2(0,T;H^1(\T^3))$. Applying the same argument to
\eqref{lemma:non-negativity_smooth_NPEuler_2}, one can get that $\Theta\in L^2(0,T;H^1(\T^3))$. Since $\rho=\Theta-\th$ and $\s=\Theta+\th$, it follows that $\rho,\s\in L^2(0,T;H^1(\T^3))$ and thus the solution is parabolic.
\end{proof}

\begin{Corollary}\label{Coro:non-negativity_NPNS}
Let $(u,\rho,\s)$ be a weak solution of of \eqref{eq:NPNS} on $[0,T]$. Suppose that
\[
\rho, \s \in L^r (0,T; L^s(\T^3))
\]
for some $r,s$ satisfying
\begin{equation}\label{Coro:non-negativity_NPNS_1}
\frac{2}{r} + \frac{3}{s} =1,
\end{equation}
where $2 \leq r \leq \infty$ and $3 \leq s \leq \infty$. Then the solution is parabolic. Furthermore, the ion concentrations $\th$ and $\Theta$ defined in \eqref{def:ion_concentration} remain non-negative provided that their initial data are non-negative.
\end{Corollary}

\begin{proof}
Recall that from Definition~\ref{def:weak_solution_NPNS},
\[
u \in L^\infty([0,T];L^2(\T^3)) \cap L^2 ([0,T];H^1(\T^3)),\quad \grad \Phi \in L^\infty([0,T];H^1(\T^3)).
\]
By the interpolation inequality we deduce that
\begin{equation*}
u, \grad \Phi \in L^p([0,T];L^q(\T^3))
\end{equation*}
for all $p$ and $q$ satisfying
\begin{equation}\label{Proof:Coro_non-negativity_NPNS_1}
\frac{2}{p} + \frac{3}{q} = \frac{3}{2} \ \ \ \text{and} \ \ \ 2 \leq q \leq 6.
\end{equation}
In particular, if $r$ and $s$ are as in \eqref{Coro:non-negativity_NPNS_1}, then the pair $(p,q)$, given with the usual convention at the endpoints by
\[
p = \frac{2r}{r-2} \ \ \ \text{and} \ \ \ q = \frac{2s}{s-2},
\]
satisfies \eqref{Proof:Coro_non-negativity_NPNS_1}. Moreover, \eqref{Thm:Regularity_NPNS_1} is satisfied by the corresponding pairs $(p,r)$ and $(q,s)$. Hence, by Theorem \ref{Thm:Regularity_NPNS}, the weak solution $(u,\rho,\s)$ is parabolic. As $u \in L^2([0,T];H^1(\T^3))$, the non-negativity of the ion concentrations follows by a density argument and the same truncation proof as in Lemma \ref{lemma:non-negativity_smooth_NPEuler}, applied separately to $\th^-$ and $\Theta^-$ in the equations \eqref{lemma:non-negativity_smooth_NPEuler_1} and \eqref{lemma:non-negativity_smooth_NPEuler_2}.
\end{proof}

\section*{Acknowledgments}

R.H. acknowledges partial support from the Simons Foundation (MP-TSM-00002783), ONR grant \#N00014-24-1-2432, and NSF grant DMS-2420988. Q.L. acknowledges partial support from the Simons Foundation (SFI-MPS-TSM-00013384).

\bibliographystyle{plain}
\bibliography{NPEuler_Onsager}

\end{document}